\documentclass[11pt]{article}

\usepackage[a4paper, total={6in, 8in}]{geometry}
\usepackage{setspace}
\usepackage{amsmath}
\usepackage{amsfonts}
\usepackage{amssymb}
\usepackage{mathtools}
\usepackage[english]{babel}
\usepackage{amsthm}
\usepackage{tikz-cd}

\usepackage{a4wide}

\usepackage{enumitem}
\usepackage{titling}

\usepackage{xcolor}

\DeclareMathOperator{\Ann}{Ann}
\DeclareMathOperator{\lAnn}{l.Ann}
\DeclareMathOperator{\rAnn}{r.Ann}
\DeclareMathOperator{\Z}{Z}
\DeclareMathOperator{\J}{J}

\DeclareMathOperator{\C}{C}
\DeclareMathOperator{\N}{N}

\DeclareMathOperator{\id}{id}

\DeclareMathOperator{\soc}{soc}

\newtheorem{corollary}{Corollary}[section]
\newtheorem{lemma}[corollary]{Lemma}
\newtheorem{theorem}[corollary]{Theorem}
\newtheorem{proposition}[corollary]{Proposition}

\theoremstyle{definition}
\newtheorem{definition}[corollary]{Definition}
\newtheorem{numberedexample}[corollary]{Example}
\newtheorem{problem}[corollary]{Problem}
\newtheorem{remark}[corollary]{Remark}

\date{}

\begin{document}
\title{On Minimal Noncommutative Rings}
\author{V. V. Bavula and N. Blacher}

\maketitle
\vspace*{-50pt}
\begin{abstract}
We study minimal noncommutative rings, that is noncommutative rings whose proper subrings and homomorphic images are all commutative. These rings were introduced by Bell and Danchev in order to test commutativity theorems. They raised the problems of describing all such rings in the finite and infinite cases. In the finite case, we give a classification into three pairwise disjoint classes, the first two of which are completely characterised. For the third class, we give a finite procedure which can produce any of (and only) the required rings. We also translate the problem into commutative algebra, in terms of finite local rings with small socle and a kind of cancellation property. Finally, we show that if an infinite minimal noncommutative ring exists, then it is a division algebra with very strange properties, and a counterexample to several longstanding conjectures.
\end{abstract}

\section{Introduction and notation}

In 1945, Jacobson \cite{X^n-X} proved that in a ring $R$ if each element $a$ satisfies $a^{n(a)}=a$ for some integer $n(a)>1$ then $R$ is commutative. This celebrated result inspired many further works on conditions forcing a ring to be commutative. Most notably, Herstein generalised Jacobson's result \cite{Herstein_GeneralisationX^n-X} and discovered many more polynomial identities which force commutativity. Different conditions were studied extensively, in particular by H. Bell. We refer the reader to \cite{PINTERLUCKE2007165} for a history of commutativity theorems.

More recently, a new approach was proposed by J. Bell and P. Danchev \cite{BellandDanchev}, inspired by the work of Belov-Kanel, Rowen and Vishne on affine representability \cite{MR4763155}. They showed \cite[Theorem~2.2]{BellandDanchev} that given sets $\mathcal{S}$ and $\mathcal{T}$ of polynomial identities, with $\mathcal{T}$ finite, if there exists a ring $R$ such that all elements of $\mathcal{S}$ are identities for $R$ and all elements of $\mathcal{T}$ are non-identities for $R$, then there exists a finite ring with the same two properties.

From the perspective of commutativity theorems, taking $\mathcal{T}=\lbrace XY-YX\rbrace$, the above result means that if there exists a noncommutative ring satisfying all the identities of a set $\mathcal{S}$, then there exists a finite noncommutative ring with the same property \cite[Theorem~2.4]{BellandDanchev}. Furthermore, subrings and homomorphic images inherit polynomial identities. This easy observation motivates the following definition.

\begin{definition}
A ring $R$ is called \textit{minimal noncommutative} if
\begin{itemize}[topsep=0pt]
\item[(1)] $R$ is not commutative, and
\item[(2)] every proper subring and homomorphic image of $R$ is commutative.
\end{itemize}
\end{definition}

Thus, to determine whether a set of polynomial identities forces the ring to be commutative, one need only check finite minimal noncommutative rings. In \cite[Theorem~2.5]{BellandDanchev} Bell and Danchev gave a coarse classification, which allowed them to produce an algorithm that decides whether a finite set of polynomial identities forces commutativity (see \cite[Section~3]{BellandDanchev}).

Their classification is divided into three classes. However, as they point out in \cite[Remark~2.6]{BellandDanchev}, not all rings that appear in these classes are minimal noncommutative, and they raise the question of giving a precise classification of finite minimal noncommutative rings. We approach this problem in Sections~\ref{FiniteCaseandIdentities},~\ref{IterativeProcedure}~and~\ref{Abelianisation}.

In Proposition~\ref{FirstTwoClasses} we refine the easier first two classes, excluding all superfluous rings. This yields the refined classification of Theorem~\ref{RefinedClassification} into three pairwise disjoint classes. Our focus then turns to the third class, for which subrings are always commutative (Proposition~\ref{SubringsInA_pCommutative}). Hence, only homomorphic images have to be considered, which leads us to presenting an iterative procedure from which any of (and only) the desired rings can be produced (Theorem~\ref{ProcedureAllAndOnly}). We use it to construct examples that illustrate the diversity found in the third class. We then describe a second approach, by considering abelianisations. We show in Theorem~\ref{MinimalRingGivesMGPair} that the classification problem can be translated entirely into commutative algebra, in terms of finite local rings with small socle and a kind of cancellation property (Proposition~\ref{IdealsInMinimalPairs}).

The question is raised in \cite[Remark~2.6]{BellandDanchev} whether a minimal noncommutative ring is necessarily finite. We study this problem in Section~\ref{InfiniteCase}. In Theorem~\ref{C22July22} we prove that an infinite minimal noncommutative ring must be a division algebra. Furthermore, we show that such a division algebra would exhibit very strange properties, which suggests that it cannot exist. It would constitute a counterexample to several important conjectures on division algebras, notably problems by Barbaumov (Remark~\ref{RemarkBarbaumov}), Latyshev (Remark~\ref{RemarkLatyshev}), 
Cohn (Proposition~\ref{AlgebraicIsCentral}), and the Makar-Limanov conjecture (Remark~\ref{RemarkMakar-Limanov}).

All rings considered in this paper are associative with $1$. Subrings of a given ring $R$ must contain the identity element of $R$. By a noncommutative ring we mean a ring which is not commutative. We write $F_q$ for the Galois field of order $q$. We denote inclusion by $\subset$ and strict inclusion by~$\subsetneq$. 

\section{Finite minimal noncommutative rings and their identities}\label{FiniteCaseandIdentities}

Our starting point is the coarse classification given by Bell and Danchev. They showed in \cite[Theorem~2.5]{BellandDanchev} that a finite minimal noncommutative ring must belong to one of the following three classes.

\begin{itemize}[topsep=0pt]
\item[(1)] For a prime $p$, the ring $U_p$ of upper triangular $2\times 2$ matrices over $F_p$, i.e.
\[U_p=\left\lbrace\begin{pmatrix}
a&b\\
0&c
\end{pmatrix}
:a,b,c\in F_p\right\rbrace.\]
\item[(2)] For a prime $p$, an integer $n\geq 2$ and $1\leq i\leq n-1$, the ring
\[B_{p,n,i}:=\left\lbrace\begin{pmatrix}
x^{p^i}&y\\
0&x
\end{pmatrix}
:x,y\in F_{p^n}\right\rbrace.\]
\item[(3)] For a prime $p$ and an integer $n\geq3$, the noncommutative homomorphic images of
\[A\coloneqq \mathbb{Z}\langle x,y\rangle/\left(I+(p,x,y)^n\right)\]
where $I=(p,x,y)[x,y]\mathbb{Z}\langle x,y\rangle+\mathbb{Z}\langle x,y\rangle[x,y](p,x,y)$.
\end{itemize}

The three classes in this classification correspond to the three classes in \cite[\S 3]{Redei_OneStepNoncommutativeFinite}, where finite non-unital noncommutative rings whose proper subrings are all commutative were classified. However, as pointed out in \cite[Remark~2.6]{BellandDanchev}, not all rings in these three classes are minimal noncommutative. They raised the question of giving a more precise classification of finite minimal noncommutative rings, which we set out to do. We start with the easier first two classes.

\begin{proposition}\label{FirstTwoClasses}
For any prime $p$, the following hold.
\begin{itemize}[topsep=0pt]
\item[(1)] The ring $U_p$ is minimal noncommutative.
\item[(2)] The ring $B_{p,n,i}$ is minimal noncommutative if and only if $n=q^m$ and $i=jq^{m-1}$ for some prime $q$ and integers $m\geq 1$ and $1\leq j\leq q-1$.
\end{itemize}
\end{proposition}

\begin{proof}
Checking (1) is an easy exercise. We turn to (2) and first prove the condition is necessary. We will show that $B_{p,n,i}$ has a noncommutative proper subring if either $n$ is not a prime power, or $n=q^m$ but $i$ is not of the form $jq^{m-1}$. In either case, there is a natural number $1\leq k<n$ that divides $n$ but does not divide $i$. In particular, the field $F_{p^n}$ contains a proper subfield $F_{p^k}$. Then the ring $B_{p,k,i}=\left\lbrace\begin{pmatrix}
x^{p^i}&y\\
0&x
\end{pmatrix}
:x,y\in F_{p^k}\right\rbrace$ is a proper subring of $B_{p,n,i}$. We will show that $B_{p,k,i}$ is not commutative. Considering the commutator
\[\left[\begin{pmatrix}
x^{p^i}&0\\
0&x
\end{pmatrix},\begin{pmatrix}
0&1\\
0&0
\end{pmatrix}\right]=\begin{pmatrix}
0&x^{p^i}-x\\
0&0
\end{pmatrix}\]
it suffices to show that there is an element $x\in F_{p^k}$ such that $x^{p^i}\neq x$. Since $k$ does not divide $i$ we can write $i=\alpha k+\beta$ with $0<\beta<k$. Pick $x\in F_{p^k}$ such that $x^{p^\beta}\neq x$. Such an $x$ exists since the lowest degree monic polynomial identity in one variable satisfied by $F_{p^k}$ is $X^{p^k}-X$. Now
\[x^{p^{\alpha k}}=\left(x^{p^k}\right)^{p^{(\alpha-1)k}}=x^{p^{(\alpha-1)k}}=\cdots=x^{p^k}=x\]
and therefore
\[x^{p^i}=x^{p^{\alpha k+\beta}}=x^{p^{\alpha k}p^\beta}=\left(x^{p^{\alpha k}}\right)^{p^\beta}=x^{p^\beta}\neq x\]
as required.

Now we turn to the opposite implication. That is we show that $R=B_{p,n,i}$ is a minimal noncommutative ring when $n=q^m$ is a prime power and $i$ is of the form $jq^{m-1}$ for $1\leq j\leq q-1$. That $R$ is not commutative follows as above from the fact that $X^{p^i}-X$ is not an identity for the field $F_{p^n}$. We start by showing that every proper homomorphic image $R/I$ is commutative. Since every commutator in $R$ is a strictly upper triangular matrix, it is enough to show that every matrix of the form $\begin{pmatrix}
0&z\\
0&0
\end{pmatrix}$ is contained in $I$. Pick a nonzero element $\begin{pmatrix}
x^{p^i}&y\\
0&x
\end{pmatrix}$ of $I$. If $x=0$ then $y\neq 0$ and $I$ contains
\[\begin{pmatrix}
0&y\\
0&0
\end{pmatrix}\begin{pmatrix}
\left(y^{-1}z\right)^{p^i}&0\\
0&y^{-1}z
\end{pmatrix}=\begin{pmatrix}
0&z\\
0&0
\end{pmatrix}\]
as required. If $x\neq 0$ then $I$ contains
\[\begin{pmatrix}
x^{p^i}&y\\
0&x
\end{pmatrix}\begin{pmatrix}
0&x^{-p^i}z\\
0&0
\end{pmatrix}=\begin{pmatrix}
0&z\\
0&0
\end{pmatrix}\]
as required. Thus $I$ contains $[R,R]$, i.e. $R/I$ is commutative. It remains to prove that every proper subring of $R$ is commutative. Let $S$ be a subring of $R$. Consider the homomorphism $\pi:S\to F_{p^n}$ mapping each matrix in $S$ to its bottom-right entry. The image of $\pi$ must be a subfield $F_{p^k}$ of $F_{p^n}$, while its kernel must be an $F_{p^k}$-subspace of $F_{p^n}$.

First we consider the case when $\pi(S)=F_{p^n}$. Then either $\ker(\pi)=0$ and $S$ is commutative, or $\ker(\pi)\cong F_{p^n}$ and $S=R$. So we can now assume that $\pi(S)=F_{p^k}$ is a proper subfield of $F_{p^n}$. Since $n=q^m$, this means $k=q^l$ with $0\leq l\leq m-1$. In this case, since $i=jq^{m-1}$ we have $i=kh$ with $h\geq 1$. Therefore for any $x\in F_{p^k}$ we have
\[x^{p^i}=x^{p^{kh}}=\left(x^{p^k}\right)^{p^{k(h-1)}}=x^{p^{k(h-1)}}=\cdots=x^{p^k}=x.\]
Thus every element of $S$ is of the form $\begin{pmatrix}
x&y\\
0&x
\end{pmatrix}$ and it follows that $S$ is commutative, completing the proof.
\end{proof}

Combined with the fact that the three classes are pairwise disjoint, this yields our refined classification.

\begin{theorem}\label{RefinedClassification}
A ring is finite minimal noncommutative if and only if it belongs to one of the following three pairwise disjoint classes.
\begin{itemize}[topsep=0pt]
\item[(1)] For a prime $p$, the ring $U_p$ of upper triangular $2\times 2$ matrices over $F_p$, i.e.
\[U_p=\left\lbrace\begin{pmatrix}
a&b\\
0&c
\end{pmatrix}
:a,b,c\in F_p\right\rbrace.\]
\item[(2)] For primes $p$ and $q$, and integers $m\geq 1$ and $1\leq j\leq q-1$, the ring
\[B_{p,q^m,jq^{m-1}}:=\left\lbrace\begin{pmatrix}
x^{p^{jq^{m-1}}}&y\\
0&x
\end{pmatrix}
:x,y\in F_{p^{q^m}}\right\rbrace.\]
\item[(3)] For a prime $p$ and an integer $n\geq3$, the minimal noncommutative homomorphic images of
\[A\coloneqq \mathbb{Z}\langle x,y\rangle/\left(I+(p,x,y)^n\right)\]
where $I=(p,x,y)[x,y]\mathbb{Z}\langle x,y\rangle+\mathbb{Z}\langle x,y\rangle[x,y](p,x,y)$.
\end{itemize}
\end{theorem}

\begin{proof}
By Bell and Danchev's classification, and Proposition~\ref{FirstTwoClasses}, we only need to show that the three classes are pairwise disjoint. It is clear for the first two classes, as a ring $U_p$ has order $p^3$ while a ring $B_{p,q^m,jq^{m-1}}$ has order $p^{2q^m}$.

Now, looking at the third class, consider the ring $A$. It is clear that the ring $A/([x,y])$ is commutative, and therefore the commutator ideal $[A,A]$ coincides with $([x,y])$. Since $[x,y]$ is annihilated by the maximal ideal $(p,x,y)$ both on the left and on the right, we have $[A,A]=([x,y])=F_p[x,y]$, which is central. If $R$ is a homomorphic image of $A$, then the ideal $[R,R]$ is the image of $[A,A]$, and hence $[R,R]$ must be central in $R$. Thus, in the third class, every commutator is a central element. However, in a ring $U_p$, the commutator
\[\left[\begin{pmatrix}
1&0\\
0&0
\end{pmatrix},\begin{pmatrix}
0&1\\
0&0
\end{pmatrix}\right]=\begin{pmatrix}
0&1\\
0&0
\end{pmatrix}\]
is not central. It follows that the first and third class are disjoint. Similarly, in $B_{p,q^m,jq^{m-1}}$, picking $x\in F_{p^{q^m}}$ such that $x\neq x^{p^{jq^{m-1}}}$ and any nonzero $y\in F_{p^{q^m}}$, one has
\[\left[\begin{pmatrix}
x^{p^{jq^{m-1}}}&0\\
0&x
\end{pmatrix},\begin{pmatrix}
0&y\\
0&0
\end{pmatrix}\right]=\begin{pmatrix}
0&\left(x^{p^{jq^{m-1}}}-x\right)y\\
0&0
\end{pmatrix}\]
which is not central. Therefore, the second and third classes are disjoint.
\end{proof}

We fix the notation of Theorem~\ref{RefinedClassification} for the rest of the paper. With the first two classes being settled, it remains to determine which homomorphic images of the ring $A$ are minimal noncommutative. Before turning to this problem, which turns out to be much more interesting than the first two classes, we use our refined classification to obtain some identities common to all three classes. We first collect a few observations about the third class.

\begin{lemma}\label{CentreofA}
The centre of the ring $A$ is $\Z(A)=\mathbb{Z}/p^n\mathbb{Z}+(p,x,y)^2$. 
\end{lemma}

\begin{proof}
Considering the quotient $A/(p,x,y)^2=\mathbb{Z}/p^2\mathbb{Z}\oplus x F_p \oplus y F_p$, it is clear that the centre cannot be strictly larger than what we claim. So we need only prove that the ideal $(p,x,y)^2$ is central in $A$. Take an arbitrary monomial in $(p,x,y)^2$, i.e. a product $a_1\cdots a_h$, with each $a_i\in\lbrace p,x,y\rbrace$ and $h\geq 2$. For each $1\leq j\leq h$, we have
\begin{align*}
a_1\cdots a_j x a_{j+1}\cdots a_h&=a_1\cdots [a_j, x]a_{j+1}\cdots a_h+a_1\cdots a_{j-1} x a_{j}\cdots a_h\\
&=a_1\cdots a_{j-1} x a_{j}\cdots a_h
\end{align*}
since $[a_j, x]$ is either zero or $[y,x]$, which is annihilated on both sides by any $a_i$. It follows that $a_1\cdots a_hx=xa_1\cdots a_h$ and similarly $a_1\cdots a_h$ commutes with $y$. Thus, every monomial and hence every element of $(p,x,y)^2$ is central in $A$.
\end{proof}

So the centre of $A$, and hence of every homomorphic image of $A$, is very large.

\begin{lemma}\label{A3GenOverCentre}
As a module over its centre $\Z(A)$, the ring $A$ is generated by $1,x$ and $y$. In particular, every element of $A$ can be written as $\alpha x+\beta y+z$
where $\alpha,\beta\in\lbrace 0,\ldots,p-1\rbrace$ and $z\in\Z(A)$.
\end{lemma}

\begin{proof}
This follows immediately from $A=\mathbb{Z}/p^n\mathbb{Z}+(x,y)$ and Lemma~\ref{CentreofA}.
\end{proof}

\begin{lemma}\label{CommutatorIdealofR}
Let $R$ be a homomorphic image of $A$. Then $[R,R]=F_p[x,y]$. In particular, $[R,R]$ is central and $[R,R]^2=0$.
\end{lemma}

\begin{proof}
Since a commutator in $R$ is the image of a commutator in $A$, it is enough to prove the statement for $R=A$. As we showed in the proof of Theorem~\ref{RefinedClassification}, $[A,A]=F_p[x,y]$ which is central. Moreover, since $(p,x,y)[x,y]=0$ in $A$, we obtain $[x,y]^2=0$, that is $[A,A]^2=0$.
\end{proof}

This means in particular that, despite not being rings of $2\times 2$ upper triangular matrices (unlike the first two classes), rings in the third class share polynomial identity features of such rings.

\begin{proposition}\label{CommonIdentitiesMinNonco}
Let $R$ be a finite minimal noncommutative ring. Then $R$ satisfies
\begin{itemize}[topsep=0pt]
\item[(1)] every alternating multilinear polynomial identity of degree at least 4;
\item[(2)] the identity $[X_1,X_2]X_3[X_4,X_5]=0$.
\end{itemize}
\end{proposition}

\begin{proof}
(1) It is well-known that if a ring is generated by $m$ elements as a module over its centre, then it satisfies every alternating multilinear polynomial identity of degree at least $m+1$ (see e.g. \cite[Corollary~13.1.13(i)]{McConnell-Robson}). This is the case for rings $U_p$ in the first class, and for rings in the third class by Lemma~\ref{A3GenOverCentre}. For the second class, a ring $B_{p,q^m,jq^{m-1}}$ is a subring of the ring of $2\times 2$ upper triangular matrices over $F_{p^{q^m}}$, and it inherits its identities. This includes all alternating multilinear polynomial identities of degree at least $4$, since the overring is generated by $3$ elements over its centre.

(2) In the first two classes, this follows from the fact that a commutator is a strictly upper triangular matrix. In the third class, we have $[R,R]^2=0$ by Lemma~\ref{CommutatorIdealofR}.
\end{proof}

From the perspective of commutativity-forcing sets of polynomial identities, this implies in particular that the identities of Proposition~\ref{CommonIdentitiesMinNonco} are redundant.

\begin{corollary}
Let $\mathcal{S}$ be a set of polynomial identities and let $\mathcal{S}'$ be a set of alternating multilinear polynomial identities of degree at least 4. There is a noncommutative ring which satisfies all the identities of $\mathcal{S}$ if and only if there is a noncommutative ring which satisfies all the identities of $\mathcal{S}\cup\mathcal{S}'\cup\left\lbrace[X_1,X_2]X_3[X_4,X_5]\right\rbrace$.
\end{corollary}

In addition to these identities common to all three classes, we have a family of identities for which each class satisfies a slightly different variant.

\begin{proposition}\label{CommonFamiliesIdentities}
(1) A ring $U_p$ in the first class satisfies the identities
\begin{align*}
\left(X_1^p-X_1\right)X_2[X_3,X_4]&=0\tag{1.a}\\ [X_1,X_2]X_3\left(X_4^p-X_4\right)&=0\tag{1.b}\\
\left(X_1^p-X_1\right)X_2\left(X_3^p-X_3\right)&=0\tag{1.c}\\
\intertext{(2) A ring $B_{p,q^m,jq^{m-1}}$ in the second class satisfies the identities}
\left(X_1^{p^{q^{m}}}-X_1\right)X_2[X_3,X_4]&=0\tag{2.a}\\ [X_1,X_2]X_3\left(X_4^{p^{q^{m}}}-X_4\right)&=0\tag{2.b}\\
\left(X_1^{p^{q^{m}}}-X_1\right)X_2\left(X_3^{p^{q^{m}}}-X_3\right)&=0\tag{2.c}\\
\intertext{(3) A ring $R$ in the third class for prime $p$ and integer $n\geq 3$ satisfies the identities}
\left(X_1^i-X_1\right)X_2[X_3,X_4]&=0\tag{3.a}\\ [X_1,X_2]X_3\left(X_4^i-X_4\right)&=0\tag{3.b}\\
\left(X_1^i-X_1\right)X_2\left(X_3^i-X_3\right)X_4\cdots X_{2n-2}\left(X_{2n-1}^i-X_{2n-1}\right)&=0\tag{3.c}\\
[[X_1,X_2],X_3]&=0\tag{3.d}
\intertext{ where $i=p^{n-1}(p-1)+1$.}
\end{align*}
\end{proposition}

\begin{proof}
The identities (1.a), (1.b) and (1.c) hold in $U_p$ because every commutator is a strictly upper triangular matrix, and so is $x^p-x$ for every matrix $x\in U_p$.

Similarly, the identities (2.a), (2.b) and (2.c) hold in $B_{p,q^m,jq^{m-1}}$ because every commutator is a strictly upper triangular matrix, and so is $x^{p^{q^m}}-x$ for every matrix $x\in B_{p,q^m,jq^{m-1}}$.

For (3), Lemma~\ref{CommutatorIdealofR} implies that (3.d) holds. It suffices to show that the ring $A$ satisfies the identities (3.a), (3.b) and (3.c). Every element $a$ in $A$ can be written as $a=\alpha + w$ where $0\leq \alpha\leq p-1$ and $w\in (p,x,y)$. Now, Euler's phi function yields $1=\alpha^{\phi(p^n)}=\alpha^{p^{n-1}(p-1)}=\alpha^{i-1}$ modulo $p^n$. Since the characteristic of $A$ is $p^n$, we have $\alpha^i=\alpha$ in $A$. Thus,
\[a^i=(\alpha+w)^i=\alpha^i+w'=\alpha+w'\]
for some $w'$ in $(p,x,y)$. It follows that $a^i-a=w'-w$ belongs to $(p,x,y)$. In particular, $(a^i-a)[A,A]$ is contained in $(p,x,y)[A,A]$. By Lemma~\ref{CommutatorIdealofR} this is equal to $(p,x,y)[x,y]F_p$, which is zero by definition of $A$. This proves (3.a), and similarly $[A,A](a^i-a)=0$, i.e. (3.b) also holds. Moreover, since we have shown that $a^i-a$ belongs to the ideal $(p,x,y)$, it follows that any evaluation of the polynomial
\[\left(X_1^i-X_1\right)X_2\left(X_3^i-X_3\right)X_4\cdots X_{2n-2}\left(X_{2n-1}^i-X_{2n-1}\right)\]
in $A^{2n-1}$ will yield an element of $(p,x,y)^n=0$.
\end{proof}

The identities of Propositions~\ref{CommonIdentitiesMinNonco}~and~\ref{CommonFamiliesIdentities} are the general forms of those that appear in \cite[Proposition~5.2]{BellandDanchev}, where they were shown to generate all the redundant identities in the special case of \textit{homogeneous multilinear} polynomials forcing commutativity. Propositions~\ref{CommonIdentitiesMinNonco}~and~\ref{CommonFamiliesIdentities} then constitute a partial generalisation to the general case.

Though the first two classes of Theorem~\ref{RefinedClassification} are easy to understand, the third one requires more study, as most homomorphic images of the ring $A$ are not minimal noncommutative. The definition of minimal noncommutative rings requires commutativity for both subrings and homomorphic images. We will see that in the third class the condition on subrings is always satisfied.

\begin{lemma}\label{ATwistedPolynomial}
Every element of $A$ can be uniquely written as
\[\alpha yx +\sum_{k=0}^{n-1}\sum_{0\leq s+t\leq k}\alpha_{s,t}^{(k)}p^sx^ty^{k-s-t}\]
where $\alpha,\alpha_{s,t}^{(k)}\in\lbrace 0,\ldots,p-1\rbrace$.
\end{lemma}

\begin{proof}
It suffices to show that, apart from $yx$, every monomial in $A$ in the variables $x$ and $y$ is of the form $x^ay^b$. If not, take a counterexample monomial $g$, i.e. a monomial $g$ in the ideal $(x,y)^3$ of $A$, such that every way to write $g$ in $A$ contains $yx$. Write
\[g=x^uy^vxh\]
where $u\geq 0$ is taken as large as possible, $v\geq 1$ and $h$ is a monomial. Since $g$ has degree at least 3, then either $u+v\geq 2$ or $h$ has degree at least one. In the first case, $x^uy^v$ is central by Lemma~\ref{CentreofA} and hence $g=x^uy^vxh=xx^uy^vh=x^{u+1}y^vh$. Since every expression equal to $g$ contains $yx$, the monomial $h$ must contain an $x$. But then $g=x^{u+1}y^vh$ contradicts our choice of $u$. In the second case where $h$ has degree at least one then $xh$ is central and $g=x^uy^vxh=x^uxhy^v=x^{u+1}hy^v$ which again contradicts our choice of $u$. It follows that such a $g$ does not exist, as required.
\end{proof}

\begin{proposition}\label{SubringsInA_pCommutative}
Let $R$ be a homomorphic image of $A$ and $T$ be a proper subring of $R$. Then $T$ is commutative.
\end{proposition}

\begin{proof}
The preimage of $T$ under the natural projection of $A$ onto $R$ is a proper subring of $A$, so it is enough to show the statement for $R=A$. Let $S$ be a noncommutative subring of $A$. We will show that $S=A$. Take two noncommuting elements $a_1,a_2$ in $S$. As in Lemma~\ref{A3GenOverCentre}, write
\[a_1=\alpha_1 x +\beta_1 y+z_1\]
where $\alpha_1,\beta_1\in\lbrace0,\ldots,p-1\rbrace$ and $z_1\in\Z(A)$. By Lemma~\ref{CentreofA}, $\Z(A)=\mathbb{Z}/p^n\mathbb{Z}+(p,x,y)^2$ and since $S$ contains $\mathbb{Z}/p^n\mathbb{Z}$, we can assume without loss of generality that $z_1\in(p,x,y)^2$. Similarly, write
\[a_2=\alpha_2x+\beta_2y+z_2\]
with $\alpha_2,\beta_2\in\lbrace0,\ldots p-1\rbrace$ and $z_2\in (p,x,y)^2$. Since
\[0\neq [a_1,a_2]=(\alpha_1\beta_2-\alpha_2\beta_1)[x,y]\]
we can further assume without loss of generality that $\alpha_1\neq 0$ and $\beta_2\neq 0$.

The natural numbers $\alpha_1$ and $p^n$ being coprime, there are integers $u$ and $v$ such that $u\alpha_1=1+vp^n$. Then $S$ contains
\[ua_1=u\alpha_1 x+u\beta_1y+uz_1=(1+vp^n)x+u\beta_1y+uz_1=x+u\beta_1y+uz_1\]
i.e. we can take $\alpha_1=1$. Similarly, without loss of generality we can take $\beta_2=1$. Then $S$ contains
\[a_1-\beta_1a_2=x+\beta_1y+z_1-\beta_1(\alpha_2x+y+z_2)=(1-\beta_1\alpha_2)x+(z_1-\beta_1z_2).\]
Since $0\neq [a_1,a_2]=(1-\alpha_2\beta_1)[x,y]$ and $p[x,y]=0$, we know that $p$ does not divide $(1-\beta_1\alpha_2)$. So these two numbers are coprime and there exists some integer $w$ such that $w(1-\beta_1\alpha_2)=1 \mod p^n$. Then $S$ contains
\[w(a_1-\beta_1a_2)=w(1-\beta_1\alpha_2)x+w(z_1-\beta_1z_2)=x+w(z_1-\beta_1z_2).\]
In particular, $S$ contains an element of the form $x+s_1$ where $s_1\in (p,x,y)^2$. Similarly, looking at $a_2-\alpha_2a_1$ yields an element $y+s_2$ in $S$ with $s_2\in(p,x,y)^2$.

Now, we assume that $S$ is a proper subring of $A$ and seek a contradiction. Since $A$ is generated as a ring by $x$ and $y$, the ideal $(p,x,y)$ cannot be contained in $S$. On the other hand, $(p,x,y)^n=0$ is in $S$ so we can pick $k\geq 1$ such that $(p,x,y)^k\not\subset S$ but $(p,x,y)^{k+1}\subset S$.

Pick a monomial $r$ in the variables $p,x$ and $y$ in $(p,x,y)^k\setminus S$. By Lemma~\ref{ATwistedPolynomial}, $r$ is of the form $p^ix^jy^{k-i-j}$ unless $r=yx$. We first deal with the latter case. If $r=yx$ then $k=2$ which means $(p,x,y)^3\subset S$. Recall that $S$ contains $x+s_1$ and $y+s_2$ with $s_1,s_2\in(p,x,y)^2$. Then $S$ contains
\[(y+s_2)(x+s_1)=yx+(ys_1+s_2x+s_2s_1).\]
The second summand belongs to $(p,x,y)^3$ and hence to $S$. Therefore $r=yx$ belongs to $S$, contradicting our choice of $r$. Thus only the case where $r=p^ix^jy^{k-i-j}$ remains. Now $S$ contains
\[p^i(x+s_1)^j(y+s_2)^{k-i-j}=p^ix^jy^{k-i-j}+q=r+q\]
where $q$ is contained in $(p,x,y)^{k+1}$, using again that $s_1,s_2\in(p,x,y)^2$. Our choice of $k$ then implies that $q\in S$. Therefore $r$ belongs to $S$ which is absurd. This is the desired contradiction and it shows that $S=A$.
\end{proof}

This has the following immediate consequences.

\begin{corollary}\label{OnlyHomoImages}
Let $R$ be a noncommutative homomorphic image of $A$. The following are equivalent:
\begin{itemize}[topsep=0pt]
\item[(i)] $R$ is minimal noncommutative;
\item[(ii)] every proper homomorphic image of $R$ is commutative;
\item[(iii)] every nonzero two-sided ideal of $R$ contains $[x,y]$.
\end{itemize}
\end{corollary}

\begin{proof}
By definition, (i) implies (ii). By Lemma~\ref{CommutatorIdealofR}, an ideal $I$ contains $[x,y]$ if and only if it contains $[R,R]$, which is equivalent to $R/I$ being commutative. Thus (ii) and (iii) are equivalent. That (ii) implies (i) follows from Proposition~\ref{SubringsInA_pCommutative}.
\end{proof}

So finding minimal noncommutative rings in the third class boils down to finding those homomorphic images of $A$ with $F_p[x,y]$ as their unique minimal ideal. We will show how this can be done iteratively.

\section{Iterative procedure}\label{IterativeProcedure}

\textit{Step 0:} We start with $R_0:=A$ and define
\[A_0\coloneqq\left(\lAnn_{R_0}(p,x,y)\right)\cap\left(\rAnn_{R_0}(p,x,y)\right)=F_p[x,y]+(p,x,y)^{n-1}.\]
Then for each $0\neq w\in A_0$, the ideal generated by $w$ is $F_p w$. Following Corollary~\ref{OnlyHomoImages}, we want these minimal ideals to be contained in $F_p[x,y]$. This will be done in three steps. 

\begin{itemize}[topsep=0pt]
\item[(a)] Extend $\left\lbrace[x,y]\right\rbrace$ to a basis $\left\lbrace[x,y]\right\rbrace\sqcup V_0$ for the $F_p$-vector space $A_0$.

\item[(b)] For each $v\in V_0$, pick a scalar $m_v\in F_p$ and let $I_0$ be the ideal $(v-m_v[x,y])_{v\in V_0}.$

\item[(c)] We then set $R_1\coloneqq R_0/I_0$.
\end{itemize}

By factoring out $I_0$, we unified all the minimal ideals of $R_0$ into $F_p[x,y]$. However doing so may have created new minimal ideals, and we must therefore repeat the process. 

\textit{Step 1:} Take
\[A_1\coloneqq\left(\lAnn_{R_1}(p,x,y)\right)\cap\left(\rAnn_{R_1}(p,x,y)\right)\subset F_p[x,y]+(p,x,y)^{n-2}.\]

\begin{itemize}[topsep=0pt]
\item[(a)] Extend $\left\lbrace [x,y]\right\rbrace$ to a basis $\left\lbrace [x,y]\right\rbrace\sqcup V_1$ for $A_1$. 

\item[(b)] Again, for each $v\in V_1$, pick a scalar $m_v\in F_p$ and let $I_1$ be the ideal $(v-m_v[x,y])_{v\in V_1}$ of $R_1$.

\item[(c)] We then set $R_2\coloneqq R_1/I_1$.
\end{itemize}

\textit{Step $N$:} Repeating the process yields a chain of ideals $I_0\subset I_1\subset\ldots$ which must stop, after $N$ steps say. The corresponding chain of homomorphic images
\[A=R_0\twoheadrightarrow R_1\twoheadrightarrow \ldots\twoheadrightarrow R_N\]
stops at a ring $R_N$ which satisfies 
\[\left(\lAnn_{R_N}(p,x,y)\right)\cap\left(\rAnn_{R_N}(p,x,y)\right)=F_p[x,y].\]

\begin{theorem}\label{ProcedureAllAndOnly}
A homomorphic image $R$ of $A$ is minimal noncommutative if and only if it can be obtained through the iterative procedure above.
\end{theorem}

\begin{proof}
We start by showing that a ring obtained as above is minimal noncommutative. We first show that $R\coloneqq R_N=R_{N-1}/I_{N-1}$ is not commutative. Otherwise, the commutator $[x,y]$ belongs to $I_{N-1}$. Take $k$ minimal such that $[x,y]$ belongs to $I_k$. Clearly $k\geq 1$. Then in the ring $R_k=R_{k-1}/I_{k-1}$, we have
\[[x,y]\in I_k=(v-m_v[x,y])_{v\in V_k}=\sum_{v\in V_k}F_p(v-m_v[x,y])\]
where the last equality follows from the fact that each generator of $I_k$ annihilates $(p,x,y)$ on both sides. Write $[x,y]=\sum_{v\in V_k}\alpha_v(v-m_v[x,y])$. Then in the ring $R_k$
\[\left(1+\sum_{v\in V_k}\alpha_v m_v\right)[x,y]-\sum_{v\in V_k}\alpha_vv=0.\]
But by construction, $\lbrace[x,y]\rbrace\sqcup V_k$ is a basis for the $F_p$-vector space $A_k$. In particular the generators must be linearly independent and hence $\alpha_v=0$ for all $v\in V_k$. This implies that $[x,y]=0$ in $R_k$, i.e. $I_{k-1}$ contains $[x,y]$, which is absurd by our choice of $k$. This contradiction shows that $R$ is not commutative.

By Corollary~\ref{OnlyHomoImages}, the ring $R$ is then minimal noncommutative if and only if each of its nontrivial ideals contains $[x,y]$. We will show that $[x,y]$ is indeed contained in the ideal generated by any element $0\neq w\in R$. Since $(p,x,y)^n=0$, we can take $1\leq t\leq n$ minimal such that
\[(p,x,y)^tw=0.\]
By minimality of $t$ there is $r\in(p,x,y)^{t-1}$ such that $rw\neq 0$. Similarly, we take $1\leq t'\leq n$ minimal such that
\[rw(p,x,y)^{t'}=0\]
and pick an $s\in(p,x,y)^{t'-1}$ such that $rws\neq 0$. By construction $rws$ belongs to
\[\left(\lAnn_R(p,x,y)\right)\cap\left(\rAnn_R(p,x,y)\right)=F_p[x,y]\]
where the equality holds because $R$ was obtained as the output of our iterative procedure. In particular, there is $0\neq\beta\in F_p$ such that $rws=\beta[x,y]$ and hence $[x,y]=\beta^{-1}rws$ belongs to the ideal $RwR$ as required.

It remains to prove the other implication in the statement, i.e. that every minimal noncommutative homomorphic image $R$ of $A$ can be obtained via the procedure. By Corollary~\ref{OnlyHomoImages}, $F_p[x,y]$ is the unique minimal ideal of $R$. Thus the projection $A\twoheadrightarrow R$ maps
\[A_0\coloneqq \left(\lAnn_{A}(p,x,y)\right)\cap\left(\rAnn_{A}(p,x,y)\right)=([x,y])+(p,x,y)^{n-1}\]
to $F_p[x,y]$. Extending $\left\lbrace[x,y]\right\rbrace$ to a basis $\left\lbrace[x,y]\right\rbrace\sqcup V_0$ for $A_0$, this means that every $v\in V_0$ is mapped to a (possibly zero) scalar multiple $m_v[x,y]$ where $m_v\in F_p$. In particular, $R$ is a homomorphic image of $R_1=A/\left(v-m_v[x,y]\right)_{v\in V_0}$. Repeating the process, i.e. following the procedure with scalars $m_v$ determined by $R$, we obtain a chain of homomorphic images
\[A\twoheadrightarrow R_1\twoheadrightarrow R_2\twoheadrightarrow\ldots\]
such that, by construction, $R$ is a homomorphic image of every $R_i$. As usual the chain stops, say after $N$ steps. As shown by the first part of the proof, the ring $R_N$ obtained in this way is minimal noncommutative. Therefore, since $R$ is not commutative, the homomorphism $R_N\twoheadrightarrow R$ implies that $R_N=R$ as required.
\end{proof}

In practice, the procedure is easy to follow. At each step, only the substep (a) requires some work, to find a basis for the annihilator $A_i$ of $(p,x,y)$. However, this only involves solving systems of linear equations in $F_p$. The following technical results aim at narrowing the pool of elements amongst which the basis of $A_i$ is to be found, which will yield an upper bound on the number of steps $N$ to complete the procedure. Lemma~\ref{FirstStepsSymmetric} deals with the first $n-2$ steps, Lemma~\ref{Stepn-2} deals with the $(n-1)^\text{th}$ step (if it exists), and Proposition~\ref{AtMostn-1} shows that there cannot be more steps, that is $N\leq n-1$.

\begin{lemma}\label{FirstStepsSymmetric}
At each step $R_i$ of the procedure such that $0\leq i\leq n-3$, we have
\[\lAnn_{R_i}(x)=\rAnn_{R_i}(x)\text{ and }\lAnn_{R_i}(y)=\rAnn_{R_i}(y)\]
and $\Ann_{R_i}(p,x,y)$ is contained in $(p,x,y)^{n-1-i}+F_p[x,y]$.
\end{lemma}

\begin{proof}
It is enough to show that one-sided annihilators of $x$, of $y$ and of $p$ are in the ideal $(p,x,y)^{n-1-i}+F_p[x,y]$, since the latter is central by Lemma~\ref{CentreofA}.
We proceed by induction on $i$. In $R_0=A$, each one-sided annihilator is equal to $(p,x,y)^{n-1}+F_p[x,y]$.

For $1\leq i\leq n-3$, we prove that $\rAnn_{R_i}(x)$ is contained in $(p,x,y)^{n-1-i}+F_p[x,y]$, the other cases being similar. By construction $R_i=A/I_{i-1}$ where, by the induction hypothesis, $I_{i-1}$ is contained in $(p,x,y)^{n-i}+F_p[x,y]$. Then, by Lemma~\ref{ATwistedPolynomial}, any element $a\in R_i$ can be written as
\[a=\sum_{k=0}^{n-i-2}\sum_{0\leq s+t\leq k} \alpha_{s,t}^{(k)}p^sx^ty^{k-s-t}+r\]
where $r\in (p,x,y)^{n-1-i}+F_p[x,y]$ and $\alpha_{s,t}^{(k)}\in\lbrace0,\ldots,p-1\rbrace$. Now if $xa=0$ then
\[xa=\sum_{k=0}^{n-i-2}\sum_{0\leq s+t\leq k} \alpha_{s,t}^{(k)}p^{s}x^{t+1}y^{k-s-t}+xr\in I_{i-1}\subset(p,x,y)^{n-i}+F_p[x,y].\]
Since $xr\in (p,x,y)^{n-i}$, this implies that $(p,x,y)^{n-i}$ contains
\[\sum_{k=0}^{n-i-2}\sum_{0\leq s+t\leq k} \alpha_{s,t}^{(k)}p^{s}x^{t+1}y^{k-s-t}.\]
This is only possible when each $\alpha_{s,t}^{(k)}$ is zero and therefore $a=r$. We have shown that an arbitrary element $a\in\rAnn_{R_i}(x)$ belongs to $(p,x,y)^{n-1-i}+F_p[x,y]$. The same can be done similarly for $\rAnn_{R_i}(y)$, left annihilators, and $\Ann_{R_i}(p)$, and the proof is complete.
\end{proof}

\begin{lemma}\label{Stepn-2}
If $R_{n-2}$ exists then $\left(\lAnn_{R_{n-2}}(p,x,y)\right)\cap \left(\rAnn_{R_{n-2}}(p,x,y)\right)$ is contained in the ideal $(p)+(x,y)^2$.
\end{lemma}

\begin{proof}
If the procedure has not ended before producing $R_{n-2}$ then we have $R_{n-2}=A/I_{n-3}$ for an ideal $I_{n-3}$ which is contained in $(p,x,y)^2$ by Lemma~\ref{FirstStepsSymmetric}. Then the annihilator
\[A_{n-2}:=\left(\lAnn_{R_{n-2}}(p,x,y)\right)\cap\left( \rAnn_{R_{n-2}}(p,x,y)\right)\]
is contained in $(p,x,y)$. Taking an arbitrary element $a$ in $A_{n-2}$ we can write
\[a=\alpha p+\beta x+\gamma y+r\]
with $\alpha,\beta,\gamma\in\lbrace0,\ldots,p-1\rbrace$ and $r\in (p,x,y)^2$. Then, since $a$ annihilates $x$ and $y$, we have
\[0=[x,a]=\gamma[x,y] \text{ and }0=[a,y]=\beta[x,y].\]
Since $R_{n-2}$ is not commutative it follows that $\beta=0$ and $\gamma=0$. Thus $a=\alpha p+r$ and it belongs to $(p)+(x,y)^2$ as desired.
\end{proof}

\begin{proposition}\label{AtMostn-1}
At each step $R_i$ of the procedure, the intersection
\[\Ann_{R_i}(p,x,y)\cap (p,x,y)^{n-i}\]
is either $0$ or $F_p[x,y]$. In particular, the procedure stops after at most $n-1$ steps.
\end{proposition}

\begin{proof}
Clearly the statement holds for $i=0$ as $(p,x,y)^{n}=0$. We proceed by induction on $i$. For an arbitrary element $\overline{a}\in\Ann_{R_i}(p,x,y)\cap (p,x,y)^{n-i}$, we will show that $\overline{a}$ belongs to $F_p[x,y]$. Take an element $a\in(p,x,y)^{n-i}$ in the ring $R_{i-1}$ such that $a$ is mapped to $\overline{a}$ under the natural projection $\pi:R_{i-1}\twoheadrightarrow R_i=R_{i-1}/I_{i-i}$. Since $0=p\overline{a}=\pi(pa)$, we know that $pa$ belongs to the kernel of $\pi$, that is to $I_{i-1}\subset \Ann_{R_{i-1}}(p,x,y)$. Thus $pa$ belongs to
\[\Ann_{R_{i-1}}(p,x,y)\cap (p,x,y)^{n-i+1}\]
and therefore by induction $pa=\alpha[x,y]$ for some $\alpha\in F_p$. Then
\[0=p\overline{a}=\pi(pa)=\pi(\alpha[x,y])=\alpha[x,y]\]
which implies that $\alpha=0$, i.e. $pa=0$. Similarly one can show that $xa=0$ and $ya=0$. Thus $a$ belongs to $\Ann_{R_{i-1}}(p,x,y)$, whose image under $\pi$ is $F_p[x,y]$. Therefore $\pi(a)=\overline{a}$ belongs to $F_p[x,y]$ as required.

Now, if the procedure has not ended before $n-1$ steps then it has produced a ring $R_{n-1}$. Take an arbitrary element $b$ in $\Ann_{R_{n-1}}(p,x,y)$. Write $b=\alpha + r$ with $\alpha\in\lbrace0,\ldots,p-1\rbrace$ and  $r\in (p,x,y)$. Then $0=(\alpha+r)[x,y]=\alpha[x,y]$ and hence $\alpha=0$ as $R_{n-1}$ is not commutative. Therefore
\[\Ann_{R_{n-1}}(p,x,y)=\Ann_{R_{n-1}}(p,x,y)\cap(p,x,y) =F_p[x,y]\]
and the procedure stops at $R_{n-1}$.
\end{proof}

\begin{numberedexample}\label{Generaln=3}
We now follow the procedure in the smallest case $n=3$. We start with
\[R_0=A=\mathbb{Z}\langle x,y\rangle/\left((p,x,y)^3+(p,x,y)[x,y]\mathbb{Z}\langle x,y\rangle+\mathbb{Z}\langle x,y\rangle[x,y](p,x,y)\right).\]
Then $A_0:=\Ann_{R_0}(p,x,y)=(p,x,y)^2$ and we can take as its basis
\[\left\lbrace [x,y]\right\rbrace\sqcup V_0=\left\lbrace [x,y],p^2,px,py,x^2,xy,y^2\right\rbrace.\]
Now we pick scalars
\[m_{p^2},m_{px},m_{py},m_{x^2},m_{xy},m_{y^2}\in F_p\]
and let $I_0$ be the ideal generated by $\left(v-m_v[x,y]\right)_{v\in V_0}$.

We then let $R_1:=R_0/I_0$ and look at $A_1:=\left(\lAnn_{R_1}(p,x,y)\right)\cap\left(\rAnn_{R_1}(p,x,y)\right)$. The ring $R_1$ is minimal noncommutative if and only if $A_1=F_p[x,y]$. Moreover, by Lemma~\ref{Stepn-2} we know that $A_1$ is contained in the ideal $(p)+(x,y)^2=(p)+F_p[x,y]$. Thus $R_1$ is minimal noncommutative if and only if $p$ does not annihilate $(p,x,y)$. There are two possibilities.

\begin{itemize}[topsep=0pt]
\item[(1)] $p(p,x,y)\neq 0$, i.e. at least one of $m_{p^2},m_{px}$ or $m_{py}$ is nonzero. Then $R_1$ is minimal noncommutative.
\item[(2)] $p(p,x,y)=0$, that is $0=p^2=px=py$ i.e. $0=m_{p^2}=m_{px}=m_{py}$. Then $\left\lbrace [x,y],p\right\rbrace$ is a basis for $A_1$. Picking a scalar $m_p\in F_p$ and letting $I_1$ be the ideal of $R_1$ generated by $p-m_p[x,y]$ we then obtain the ring $R_2:=R_1/I_1$. It is easily checked that the annihilator $A_2$ equals $F_p[x,y]$ i.e. $R_2$ is minimal noncommutative. This also follows from Proposition~\ref{AtMostn-1}, as the procedure must end in at most $n-1=2$ steps.
\end{itemize}

Picking for instance $m_{p^2}=1$ and $0=m_{px}=m_{py}=m_{x^2}=m_{xy}=m_{y^2}$ one obtains the minimal noncommutative ring
\[\left(\mathbb{Z}/p^3\mathbb{Z}\right)\langle x,y\rangle/\left(p^2+yx,px,py,x^2,xy,y^2\right).\]
This is our first example of a minimal noncommutative ring whose characteristic is not prime, contrary to the first two classes $U_p$ and $B_{p,q^m,jq^{m-1}}$.
\end{numberedexample}

We have shown that when $n=3$, the basis elements in $V_0\sqcup V_1$ can always be taken to be monomials. This is however not the case in general. We look at another example for which basis elements are not even homogeneous.

\begin{numberedexample}
Let $n=4$. We also assume that $p\neq 2$. As always, we start with
\[R_0=A=\mathbb{Z}\langle x,y\rangle/\left((p,x,y)^4+(p,x,y)[x,y]\mathbb{Z}\langle x,y\rangle+\mathbb{Z}\langle x,y\rangle[x,y](p,x,y)\right).\]
Then $A_0\coloneqq\Ann_{R_0}(p,x,y)=F_p[x,y]+(p,x,y)^3$ and we can take as its basis
\[\left\lbrace [x,y]\right\rbrace\sqcup V_0=\left\lbrace [x,y],p^3,p^2x,p^2y,px^2,pxy,py^2,x^3,x^2y,xy^2,y^3\right\rbrace.\]
Now we pick scalars $m_{x^3}=1,m_{x^2y}=2$ and $m_v=0$ for all the other $v\in V_0$. Let $I_0$ be the ideal generated by $\left(v-m_v[x,y]\right)_{v\in V_0}$.

We then let $R_1\coloneqq R_0/I_0$ and look at the annihilator $A_1\coloneqq \Ann_{R_1}(p,x,y)$. It is contained in $(p,x,y)^2$ by Lemma~\ref{FirstStepsSymmetric}. Clearly, the elements $[x,y],p^2,px,py$ and $y^2$ belong to $A_1$. They generate $(p,x,y)^2$ together with $x^2$ and $xy$. But if $(\alpha x^2+\beta xy)(p,x,y)=0$ then
\[\begin{cases}
(\alpha x^2+\beta xy)p=0\\
(\alpha x^2+\beta xy)x=0\\
(\alpha x^2+\beta xy)y=0
\end{cases} \text{ i.e. }\begin{cases}
\alpha m_{px^2}+\beta m_{pxy}=0\\
\alpha m_{x^3}+\beta m_{x^2y}=0\\
\alpha m_{x^2y}+\beta m_{xy^2}=0
\end{cases} \text{ i.e. }\begin{cases}
0=0\\
\alpha+2\beta=0\\
2\alpha=0
\end{cases}\]
so $\alpha=\beta=0$. Therefore no nonzero linear combination of $x^2$ and $xy$ lies in $A_1$. It follows that $A_1$ has a basis
\[\left\lbrace [x,y]\right\rbrace\sqcup V_1=\left\lbrace [x,y],p^2,px,py,y^2\right\rbrace.\]
We now pick scalars
\[0=m_{p^2}=m_{y^2}\text{ and }m_{px}=-1\text{ and }m_{py}=-2.\]
We let $I_1$ be the ideal generated by $\left(v-m_v[x,y]\right)_{v\in V_1}$ and construct the ring $R_2\coloneqq R_1/I_1$.
In $R_2$, by Lemma~\ref{Stepn-2}, the annihilator $A_2$ is contained in $(p)+(x,y)^2$ which is generated by $p,[x,y],x^2$ and $xy$. Thus, together with $[x,y]$, the basis elements for $A_2$ will be linear combinations of $p,x^2$ and $xy$. If $(\alpha p+\beta x^2+\gamma xy)(p,x,y)=0$ then
\begin{align*}
\begin{cases}
(\alpha p+ \beta x^2+\gamma xy)p=0\\
(\alpha p +\beta x^2+\gamma xy)x=0\\
(\alpha p+\beta x^2+\gamma xy)y=0
\end{cases}&\text{ i.e. }\begin{cases}
\alpha m_{p^2}+\beta m_{px^2}+\gamma m_{pxy}=0\\
\alpha m_{px}+\beta m_{x^3}+\gamma m_{x^2y}=0\\
\alpha m_{py}+\beta m_{x^2y}+\gamma m_{xy^2}=0
\end{cases}\\
&\text{ i.e. }\begin{cases}
0=0\\
-\alpha+\beta+2\gamma=0\\
-2\alpha+2\beta=0
\end{cases}
\end{align*}
i.e. $\alpha=\beta$ and $\gamma=0$. So a basis for $A_2$ is
\[\left\lbrace [x,y]\right\rbrace\sqcup V_2=\left\lbrace [x,y],p+x^2\right\rbrace.\]
Picking for instance $m_{p+x^2}=3$, we let $I_2$ be the ideal generated by $p+x^2-3[x,y]$ and construct the ring $R_3\coloneqq R_2/I_2$. By Proposition~\ref{AtMostn-1}, the procedure stops there and
\[R_3=A/\big(p+x^2-3[x,y],p^2,px+[x,y],py+2[x,y],y^2,x^3-[x,y],x^2y-2[x,y]\big)\]
is a minimal noncommutative ring. Note that it has a defining relation $p+x^2=3[x,y]$, where the left hand side is not homogeneous.
\end{numberedexample}

As we have seen, it is relatively easy to follow the procedure to construct minimal noncommutative rings in the third class. However, many different choices of scalars will actually generate isomorphic rings. The easiest example, which we will now look at, is a swap of generators $x$ and $y$.

\begin{numberedexample}\label{IsomorphicDifferentScalars}
Again we let $n=3$ and we choose scalars
\[m_{px}=1,m_p^2=m_{py}=m_{x^2}=m_{xy}=m_{y^2}=0.\]
Then as we saw in Example~\ref{Generaln=3}, the ring
\[A/\left(px-[x,y],p^2,py,x^2,xy,y^2\right)\]
is minimal noncommutative. It is however easy to see that, under the map that sends $x$ to $y$ and $y$ to $x$, it is isomorphic to the ring
\[A/\left(py+[x,y],p^2,px,x^2,xy,y^2\right)\]
given by the different choice of scalars
\[m_{py}=-1,m_{xy}=1,m_p^2=m_{px}=m_{x^2}=m_{y^2}=0.\]
More generally, under that same map, the minimal noncommutative ring $R$ given by the choice of scalars
\[m_{p^2},m_{px},m_{py},m_{x^2},m_{xy},m_{y^2}\in F_p\]
will be isomorphic to the ring $R'$ given by the choice of scalars
\begin{align*}
&m'_{p^2}=-m_{p^2},\hspace{10pt} m'_{px}=-m_{py},\hspace{10pt} m'_{py}=-m_{px},\hspace{10pt}m'_{x^2}=-m_{y^2},\\
&m'_{xy}=m'_{yx}+1=-m_{xy}+1\hspace{10pt} \text{ and }\hspace{10pt}m'_{y^2}=-m_{x^2}.
\end{align*}
\end{numberedexample}

\begin{remark}
By Theorem~\ref{ProcedureAllAndOnly}, for fixed prime $p$ and index $n$, the (finitely many) different ways of following the procedure will generate all the minimal noncommutative homomorphic images of $A$. From the point of view of classifying minimal noncommutative rings, we are of course only interested in isomorphism classes. It would be interesting to find a way to use the procedure to construct each such isomorphism class, without generating too many representatives of the same class.

As we showed in Lemma~\ref{A3GenOverCentre}, a minimal noncommutative ring in the third class is generated by three elements as a module over its centre. It would be interesting to know which commutative rings appear as the centre of a minimal noncommutative ring in the third class. Our second approach also exploits this proximity to commutative rings.
\end{remark}

\section{Abelianisation}\label{Abelianisation}

As illustrated for instance by Lemma~\ref{ATwistedPolynomial}, the ring $A$ can be seen as only one monomial away from commutative.

\begin{lemma}\label{AModuleIsomorphism}
As a module over $\mathbb{Z}/p^n\mathbb{Z}$ (i.e. as an abelian group) the direct sum
\[\mathbb{Z}[x,y]/(p,x,y)^n\bigoplus F_p\]
is isomorphic to $A$ under the map $(1,0)\mapsto 1, (x^sy^t,0)\mapsto x^sy^t$ and $(0,1)\mapsto [x,y]$.
\end{lemma}

\begin{proof}
By Lemma~\ref{ATwistedPolynomial} this map defines a homomorphism of abelian groups which is easily checked to be an isomorphism.
\end{proof}

\begin{remark}\label{Identification}
We shall repeatedly identify elements of $\mathbb{Z}[x,y]/(p,x,y)^n$ with their natural image in $A$ by composing the isomorphism of Lemma~\ref{AModuleIsomorphism} with the identity map onto the first summand, i.e.
\[\mathbb{Z}[x,y]/(p,x,y)^n\xrightarrow{\text{\tiny{\(\begin{pmatrix}
\id\\
0
\end{pmatrix}\)}}}
\mathbb{Z}[x,y]/(p,x,y)^n\bigoplus F_p\to A.\]
In order not to overcomplicate the presentation, we will do this without indicating it with extra notation.
\end{remark}

Now, let $R$ be a homomorphic image of $A$. We want to study minimal noncommutative rings $R$ via their abelianisation $R/([x,y])$. As we saw in Lemma~\ref{ATwistedPolynomial}, every monomial in $R$ is of the form $p^sx^ty^{k-s-t}$, except for $yx$. Moreover, any defining relation of $R$ of the form $yx+u=m[x,y]$ is equivalently written as $xy+u=(m+1)[x,y]$. In other words, the defining relations of $R$ can always be written in a form that does not contain $yx$.

The natural projection $R\twoheadrightarrow R/([x,y])=:C$ yields a commutative ring $C=\mathbb{Z}[x,y]/\mathfrak{a}$ where $\mathfrak{a}$ is an ideal of the ring $\mathbb{Z}[x,y]/(p,x,y)^n$. We will see that the minimal generators of the ideal $\mathfrak{a}$ correspond to the annihilator elements denoted by $v$ in the iterative procedure.

As in Remark~\ref{Identification}, one can embed $\mathfrak{a}$ into $A$. The composition of this map with the surjection of $A$ onto $R$ and then onto $C$ is zero by construction. In particular, the image of $\mathfrak{a}$ under the composition
\[\mathfrak{a}\xrightarrow{\text{\tiny{$\begin{pmatrix}
\id\\
0
\end{pmatrix}$}}} \mathbb{Z}[x,y]/(p,x,y)^n\bigoplus F_p\cong A\twoheadrightarrow R\]
lies in the kernel of the projection $R\twoheadrightarrow C$, i.e. in the ideal $F_p[x,y]$ of $R$. Thus we obtain the following homomorphism of $\mathbb{Z}/p^n\mathbb{Z}$-modules, which we call $\mu$.
\[\mu:\mathfrak{a}\to F_p[x,y]\hookrightarrow F_p\]
where the last map sends $[x,y]$ to $1$ if $[x,y]\neq 0$, and is the zero map otherwise (if and only if $R$ is commutative). The map $\mu$ determines and is determined by the scalars denoted by $m_v$ in the iterative procedure. In the following diagram, both rows are exact and the right hand side square commutes. We have built the homomorphism $\mu$ such that the left hand side square commutes too.
\[\begin{tikzcd}
\mathfrak{a} \arrow[r,hook] \arrow[d,dashrightarrow,"\mu"]& \mathbb{Z}[x,y]/(p,x,y)^n \arrow[d,hook] \arrow[r,twoheadrightarrow]& C \arrow[dd,equal]\\
F_p \arrow[d,twoheadrightarrow] & A\arrow[d,twoheadrightarrow] &\\
F_p[x,y] \arrow[r,hook] & R \arrow[r,twoheadrightarrow]& C
\end{tikzcd}
\]
We will use this construction throughout the section, e.g. by moving elements up and down the squares. We call the ideal $\mathfrak{a}$ and the map $\mu$ the \textit{pair associated} to the ring $R$.

\begin{lemma}
The ring $R$ is uniquely determined by the pair $(\mathfrak{a},\mu)$ associated to it.
\end{lemma}

\begin{proof}
We know that $R$ is uniquely determined by the natural projection $A\twoheadrightarrow R$ and hence by the surjective homomorphism
\[\mathbb{Z}[x,y]/(p,x,y)^n\bigoplus F_p\cong A\twoheadrightarrow R.\]
Let $(a_i)_{i\in I}$ be a generating set for the $\mathbb{Z}/p^n\mathbb{Z}$-module $\mathfrak{a}$. Then the sequence
\[(a_i,-\mu(a_i))_{i\in I}\hookrightarrow \mathbb{Z}[x,y]/(p,x,y)^n\bigoplus F_p\twoheadrightarrow R\]
is short exact by construction of $\mathfrak{a}$ and $\mu$.
\end{proof}

Equivalently, identifying every element $a\in \mathfrak{a}$ with its canonical image in $A$ as in Remark~\ref{Identification}, we obtain
\[R=A/\left(a-\mu(a)[x,y]\right)_{a\in\mathfrak{a}}.\]
One can see the analogy with the presentation
\[R=A/\left(v-m_v[x,y]\right)_{v\in V_0\sqcup\ldots\sqcup V_{N-1}}\]
given by the iterative procedure of Section~\ref{IterativeProcedure}.

\begin{lemma}\label{muSurface}
The restriction of $\mu$ to $\mathfrak{a}(p,x,y)$ is zero.
\end{lemma}

\begin{proof}
Take any element $a$ in $\mathfrak{a}$. As in Lemma~\ref{ATwistedPolynomial}, we write
\[a=\sum_{k=0}^{n-1}\sum_{0\leq s+t\leq k}\alpha_{s,t}^{(k)}p^sx^ty^{k-s-t}.\]
Then, by definition of $\mathfrak{a}$, we know that in $R$
\[\sum_{k=0}^{n-1}\sum_{0\leq s+t\leq k}\alpha_{s,t}^{(k)}p^sx^ty^{k-s-t}=m[x,y]\]
for some $m\in F_p$. Thus $ax$ maps to
\[\sum_{k=0}^{n-1}\sum_{0\leq s+t\leq k}\alpha_{s,t}^{(k)}p^{s}x^{t+1}y^{k-s-t}\]
in $R$ which is equal to $mx[x,y]=0$. Similarly, $ap$ and $ay$ map to $0$ and the result follows.
\end{proof}

This shows that $\mu$ is defined entirely on the minimal generators of the ideal $\mathfrak{a}$.

\begin{lemma}\label{SocleIntersectsDegree2}
Let $\mathfrak{b}$ be an ideal of $\mathbb{Z}[x,y]/(p,x,y)^n$ properly contained in $(p)+(x,y)^2$ and denote by $B$ the ring $\mathbb{Z}[x,y]/\mathfrak{b}$. Then $\Ann_B(p,x,y)=\soc\left(\mathbb{Z}[x,y]/\mathfrak{b}\right)$ intersects the ideal $(p)+(x,y)^2$ nontrivially.
\end{lemma}

\begin{proof}
Using that $(p,x,y)^n=0$ is contained in $\mathfrak{b}$, we can take $i\geq 1$ maximal such that $(p,x,y)^i\not\subset \mathfrak{b}$. We know that $(p,x,y)^i$ is contained in $\Ann_B(p,x,y)$. If $i$ is at least 2 then $(p,x,y)^i$ is contained in $(p)+(x,y)^2$, yielding the required nontrivial intersection. If $i=1$ then $(p,x,y)^2$ is contained in $\mathfrak{b}$. Then $p$ does not belong to $\mathfrak{b}$ as otherwise $\mathfrak{b}$ contains $(p)+(x,y)^2$, contradicting our choice of $\mathfrak{b}$. Then $p$ is nonzero in $B$ and it belongs to $\Ann_B(p,x,y)$.
\end{proof}

\begin{definition}\label{MinimalGenerating}
Let $p$ be a prime and $n\geq 3$ an integer. We call a pair $\left(\mathfrak{b},\nu\right)$ \textit{minimal generating} if
\begin{itemize}[topsep=0pt]
\item[(1)] $\mathfrak{b}$ is an ideal of $\mathbb{Z}[x,y]/(p,x,y)^n$ such that $(p,x,y)^{n-1}\subset\mathfrak{b}\subset(p)+(x,y)^2$,
\item[(2)] $\nu:\mathfrak{b}\to F_p$ is a homomorphism of abelian groups satisfying $\nu|_{\mathfrak{b}(p,x,y)}=0$, and
\item[(3)] for any element $u$ in the ideal $(p)+(x,y)^2$ of $\mathbb{Z}[x,y]/(p,x,y)^n$ such that $u(p,x,y)$ is contained in $\mathfrak{b}$, either $u$ is in $\mathfrak{b}$ or at least one of $\nu(pu)$, $\nu(xu)$ and $\nu(yu)$ is nonzero.
\end{itemize}
\end{definition}

We will show that minimal generating pairs correspond exactly to minimal noncommutative rings. There is however a single exception: for $n=3$, the minimal noncommutative ring $A/\left(p,x^2,xy,y^2\right)$ yields the same pair as its abelianisation $\mathbb{Z}[x,y]/\left((p)+(x,y)^2\right)$. We therefore need to exclude the latter in the following result.

\begin{theorem}\label{MinimalRingGivesMGPair}
Let $R$ be a homomorphic image of $A$ which is not $\mathbb{Z}[x,y]/\left((p)+(x,y)^2\right)$. Let
\[\mathfrak{a}\lhd \mathbb{Z}[x,y]/(p,x,y)^n \text{ and }\mu:\mathfrak{a}\to F_p\]
be the ideal and map associated to $R$. Then $R$ is minimal noncommutative if and only if $\left(\mathfrak{a},\mu\right)$ is a minimal generating pair.
\end{theorem}

\begin{proof}
First, assume that $R$ is minimal noncommutative and construct from $R$ the commutative ring $C$, the ideal $\mathfrak{a}$ and the map $\mu$ as above. We start by showing that the ideal $\mathfrak{a}$ satisfies condition (1) of Definition~\ref{MinimalGenerating}.

If $0\neq w\in(p,x,y)^{n-1}\lhd R$ then the ideal $(w)$ is a minimal ideal and, by minimality of $R$, we have $(w)=F_p[x,y]$. Thus the ideal $(p,x,y)^{n-1}\lhd R$ projects to zero under the map $R\twoheadrightarrow C$. This shows that $(p,x,y)^{n-1}\subset \mathfrak{a}$. Clearly, $\mathfrak{a}$ must be contained in $(p,x,y)$. We will show it is contained in $(p)+(x,y)^2$. That is, taking an arbitrary element $a\in(p,x,y)\lhd R$, we will show that if $a$ is mapped to zero under $R\twoheadrightarrow C$ then $a\in(p)+(x,y)^2$. We can write
\[a=\alpha p+\beta x +\gamma y+r\]
with $\alpha,\beta,\gamma\in\lbrace0,\ldots,p-1\rbrace$ and $r\in (p,x,y)^2$. Since $a$ is mapped  to zero in $C$, we know that $a\in F_p[x,y]$. Thus $0=[x,a]=\gamma[x,y]$ and $0=[a,y]=\beta[x,y]$ i.e. $\beta=\gamma=0$ as required.

By Lemma~\ref{muSurface}, the map $\mu$ satisfies condition (2) of Definition~\ref{MinimalGenerating}. Finally we show that the pair $(\mathfrak{a},\mu)$ satisfies condition (3). Take an element $u$ in the ideal $(p)+(x,y)^2$ of $\mathbb{Z}[x,y]/(p,x,y)^n$ such that $u(p,x,y)$ is in $\mathfrak{a}$ and $\mu(pu)=\mu(xu)=\mu(yu)=0$. We will show that $u$ must belong to $\mathfrak{a}$. As in Remark~\ref{Identification}, we identify $u$ with its natural image $\overline{u}$ in the ideal $(p)+(x,y)^2$ of $R$. By construction, we obtain in the ring $R$
\[x\overline{u}=\mu(xu)[x,y]=0.\]
Since $\overline{u}$ belongs to $(p)+(x,y)^2$, it is central by Lemma~\ref{CentreofA}. It follows that $\overline{u}x$ is also zero. Similarly, $p$ and $y$ annihilate $\overline{u}$ from both sides. Thus, the ideal of $R$ generated by $\overline{u}$ is either minimal or zero. Since $R$ is minimal noncommutative, this implies that $\overline{u}$ belongs to the ideal $F_p[x,y]$. Therefore $\overline{u}$ maps to zero under the map $R\twoheadrightarrow C$, i.e. $u$ belongs to $\mathfrak{a}$ as required.

We now prove that if $R$ yields a minimal generating pair $(\mathfrak{a},\mu)$, then $R$ must be minimal noncommutative. We first show that $R$ is not commutative. If $\mathfrak{a}=(p)+(x,y)^2$ this follows from the fact that we assumed that $R$ is not $\mathbb{Z}[x,y]/\left((p)+(x,y)^2\right)$. So we can assume that $\mathfrak{a}$ is properly contained in $(p)+(x,y)^2$. By Lemma~\ref{SocleIntersectsDegree2}, there exists an element $u$ in $(p)+(x,y)^2$ such that $u$ is not contained in $\mathfrak{a}$ but $u(p,x,y)$ is in $\mathfrak{a}$. Since the pair $\left(\mathfrak{a},\mu\right)$ is minimal generating, condition~(3) of Definition~\ref{MinimalGenerating} implies that one of $\mu(pu)$, $\mu(xu)$ and $\mu(yu)$ is nonzero. In particular, the map $\mu$ is not zero. However, if $R$ is commutative then $R=C=\mathbb{Z}[x,y]/\mathfrak{a}$ and the map $\mu$ is zero by construction. So $R$ is not commutative and it suffices to show that $F_p[x,y]$ is its unique minimal ideal by Corollary~\ref{OnlyHomoImages}.

Now take $w\in R$ such that the ideal $(w)\lhd R$ is minimal. We will show that the projection $R \twoheadrightarrow C$ maps $w$ to zero. We first deal with the case when $w$ does not belong to $(p)+(x,y)^2$. We can then write $w=\alpha x+\beta y + r$ for $\alpha,\beta\in\lbrace 0,\ldots,p-1\rbrace$ and $r\in (p)+(x,y)^2$ with at least one of $\alpha$ and $\beta$ not zero. Then the ideal $(w)$ contains $[w,y]=\alpha[x,y]$ and $[x,w]=\beta[x,y]$, and therefore $[x,y]$. Since $(w)$ is a minimal ideal it follows that $(w)=F_p[x,y]$ as required.

Finally, if $w$ belongs to $(p)+(x,y)^2$ then so does its image $\overline{w}$ under $R\twoheadrightarrow C$. Since $w$ generates a minimal ideal, we know that $w\in\Ann_R(p,x,y)$ and hence $\overline{w}$ belongs to $\Ann_C(p,x,y)\cap\left((p)+(x,y)^2\right)$. Take an element $u$ in the preimage of $\overline{w}$ under the natural projection $\mathbb{Z}[x,y]/(p,x,y)^n\twoheadrightarrow C$. Then $u(p,x,y)$ is in $\mathfrak{a}$ and $0=pw=xw=yw$ implies that $\mu(pu)=\mu(xu)=\mu(yu)=0$. Since the pair $\left(\mathfrak{a},\mu\right)$ is minimal generating, condition~(3) of Definition~\ref{MinimalGenerating} implies that $u$ belongs to $\mathfrak{a}$. Thus $\overline{w}$ is zero in $C$, i.e. $w$ belongs to $F_p[x,y]$ as required.
\end{proof}

As a result, given a homomorphic image of $A$, one can check whether it is minimal noncommutative purely within the bounds of commutative algebra, by constructing its associated pair. On the other hand, one can start from the commutative setting to construct any minimal noncommutative homomorphic image of $A$. Again, we implicitly use Remark~\ref{Identification} to identify elements of $\mathbb{Z}[x,y]/(p,x,y)^n$ with their natural image in $A$.

\begin{corollary}\label{MinimalPairGivesMinimalRing}
Let $(\mathfrak{b},\nu)$ be a minimal generating pair. Then the ring
\[R\coloneqq A/\left(a-\nu(a)[x,y]\right)_{a\in\mathfrak{b}}\]
is minimal noncommutative.
\end{corollary}

\begin{proof}
It is easily checked that $(\mathfrak{b},\nu)$ is the pair associated to the ring $R$ and the result follows from Theorem~\ref{MinimalRingGivesMGPair}.
\end{proof}

We can now easily construct minimal noncommutative rings of arbitrarily large characteristic.

\begin{numberedexample}
Let $p$ be a prime number and $n\geq 4$ an integer. In the ring $\mathbb{Z}[x,y]/(p,x,y)^n$, consider the ideal
\[\mathfrak{a}=\left(p^{n-1},x^{n-1},y^{n-1},px,py,xy\right).\]
Let $\mu:\mathfrak{a}\to F_p$ be the homomorphism of abelian groups given by
\[\mu(p^{n-1})=1,\qquad\mu(x^{n-1})=1,\qquad\mu(y^{n-1})=1,\qquad\mu(px)=\mu(py)=\mu(xy)=0,\]
and $\mu|_{\mathfrak{a}(p,x,y)}=0$.
We show that $(\mathfrak{a},\mu)$ is a minimal generating pair. By construction, the ideal $\mathfrak{a}$ and the map $\mu$ satisfy the conditions (1) and (2) of Definition~\ref{MinimalGenerating}, so we need only check condition (3). The socle of $\mathbb{Z}[x,y]/\mathfrak{a}$ is generated by $p^{n-2}, x^{n-2}$ and $y^{n-2}$. Thus, if $u$ is an element of $(p)+(x,y)^2$ such that $u(p,x,y)$ is contained in $\mathfrak{a}$, we can write
\[u=\alpha_1 p^{n-2}+\alpha_2 x^{n-2}+\alpha_3 y^{n-2}+a\]
for $\alpha_1,\alpha_2,\alpha_3\in\lbrace 0,\ldots,p-1\rbrace$ and $a\in\mathfrak{a}$. Then using the fact that $\mu$ is zero on the interior $\mathfrak{a}(p,x,y)$, we have
\[\mu(pu)=\alpha_1,\quad \mu(xu)=\alpha_2 \quad \text{ and }\quad \mu(yu)=\alpha_3.\]
All three are zero if and only if $\alpha_1=\alpha_2=\alpha_3=0$, i.e. $u$ belongs to $\mathfrak{a}$.

It follows that $(\mathfrak{a},\mu)$ is a minimal generating pair and therefore the ring
\[A/\left(p^{n-1}-[x,y],x^{n-1}-[x,y],y^{n-1}-[x,y],px,py,xy\right)\]
is minimal noncommutative by Corollary~\ref{MinimalPairGivesMinimalRing}. Its characteristic is $p^{n}$.
\end{numberedexample}

By Theorem~\ref{MinimalRingGivesMGPair} and Corollary~\ref{MinimalPairGivesMinimalRing}, describing minimal noncommutative rings in the third class is equivalent to describing minimal generating pairs. It is then natural to ask which ideals $\mathfrak{a}$ appear in a minimal generating pair, or in other words which commutative rings appear as $R/([x,y])$ where $R$ is a minimal noncommutative ring in the third class. In the following Proposition, we use the commutative colon notation $\left(I:J\right)\coloneqq\lbrace r:rJ\subset I\rbrace$.

\begin{proposition}\label{IdealsInMinimalPairs}
Let $\mathfrak{b}$ be an ideal of $\mathbb{Z}[x,y]/(p,x,y)^n$ for which there exists a homomorphism $\nu:\mathfrak{b}\to F_p$ such that $\left(\mathfrak{b},\nu\right)$ is a minimal generating pair. Denote by $B$ the ring $\mathbb{Z}[x,y]/\mathfrak{b}$ and by $k$ the dimension $\dim_{F_p}\soc(B)\cap\left((p)+(x,y)^2\right)$. Then
\begin{itemize}[topsep=0pt]
\item[(i)] $k\leq 3$;
\item[(ii)] $k\leq \dim_{F_p}\soc(B)\leq k+2$, and
\item[(iii)] $\left(\mathfrak{b}(p,x,y):(p,x,y)\right)=\mathfrak{b}$.
\end{itemize}
\end{proposition}

\begin{proof}
For (i), let $v_1,\ldots,v_4$ be four elements of $(p)+(x,y)^2$ such that $v_i(p,x,y)$ lies in $\mathfrak{b}$ for each $1\leq i\leq 4$. We will show that a nonzero linear combination of $v_1,\ldots,v_4$ lies in $\mathfrak{b}$. Consider the $3\times 4$ matrix over $F_p$
\[\begin{pmatrix}
\nu(pv_1)&\nu(pv_2)&\nu(pv_3)&\nu(pv_4)\\
\nu(xv_1)&\nu(xv_2)&\nu(xv_3)&\nu(xv_4)\\
\nu(yv_1)&\nu(yv_2)&\nu(yv_3)&\nu(yv_4)
\end{pmatrix}.\]
Take $\gamma_1,\gamma_2,\gamma_3,\gamma_4\in\lbrace 0,\ldots,p-1\rbrace$ not all zero such that $\left(\gamma_1,\gamma_2,\gamma_3,\gamma_4\right)^\top$ lies in its right kernel. Then, by construction, the nonzero linear combination
\[v\coloneqq\gamma_1 v_1+\gamma_2v_2+\gamma_3v_3+\gamma_4v_4\]
satisfies $\nu(pv)=\nu(xv)=\nu(yv)=0$. Since $\left(\mathfrak{b},\nu\right)$ is a minimal generating pair, condition (3) of Definition~\ref{MinimalGenerating} implies that $v$ belongs to $\mathfrak{b}$ as required.

For (ii), the first inequality is trivial as one vector space contains the other. We prove the second inequality. Take any three elements $w_1,w_2$ and $w_3$ of $\soc(B)$. We will show that a nonzero linear combination of $w_1,w_2$ and $w_3$ lies in $\soc(B)\cap\left((p)+(x,y)^2\right)$. Write
\[w_1=\alpha_1 x+\beta_1 y +r_1,\qquad w_2=\alpha_2 x+\beta_2 y +r_2,\qquad \text{ and }\qquad w_3=\alpha_3 x+\beta_3 y +r_3\]
for $\alpha_i,\beta_i\in\lbrace 0,\ldots,p-1\rbrace$ and $r_1,r_2,r_3\in (p)+(x,y)^2$. Take $\gamma_1,\gamma_2,\gamma_3\in\lbrace 0,\ldots,p-1\rbrace$ not all zero such that
\[\begin{pmatrix}
\alpha_1&\alpha_2&\alpha_3\\
\beta_1&\beta_2&\beta_3
\end{pmatrix}
\begin{pmatrix}
\gamma_1\\
\gamma_2\\
\gamma_3
\end{pmatrix}=0 \mod{p}.\]
Then $\gamma_1 w_1+\gamma_2 w_2+\gamma_3 w_3$ belongs to $(p)+(x,y)^2$ as required.

For (iii), one inclusion is trivial so we need to show that given $u$ in $\mathbb{Z}[x,y]/(p,x,y)^n$ such that $u(p,x,y)$ is contained in $\mathfrak{b}(p,x,y)$, then $u$ belongs to $\mathfrak{b}$. First, we will show that $u$ must belong to $(p)+(x,y)^2$. Since $\mathfrak{b}$ belongs to a minimal generating pair, we have $b\subset (p)+(x,y)^2$ by Definition~\ref{MinimalGenerating}~(1). Then $\mathfrak{b}(p,x,y)$ is contained in
\begin{equation}\label{InteriorInclusion}
\left((p)+(x,y)^2\right)(p,x,y)=(p^2,px,py)+(x,y)^3.
\end{equation}
Clearly $u$ belongs to $(p,x,y)$ and we write
\[u=\alpha_1 p+\alpha_2 x+\alpha_3 y+r\]
for $\alpha_1,\alpha_2,\alpha_3\in\lbrace 0,\ldots,p-1\rbrace$ and $r$ in $(p,x,y)^2$. Now $xu$ belongs to $\mathfrak{b}(p,x,y)$, and by the equality (\ref{InteriorInclusion}) this implies that $\alpha_2 x^2+\alpha_3xy$ belongs to $(p^2,px,py)+(x,y)^3$. Thus $\alpha_2$ and $\alpha_3$ are zero, i.e. $u$ belongs to $(p)+(x,y)^2$ as required. Now, $\nu(pu)=\nu(xu)=\nu(yu)=0$ since $\nu$ is zero on $\mathfrak{b}(p,x,y)$. Thus, condition~(3) of Definition~\ref{MinimalGenerating} implies that $u$ belongs to $\mathfrak{b}$ as required.
\end{proof}

\begin{corollary}\label{MoreThanOneStep}
Let $R$ be a minimal noncommutative homomorphic image of $A$ for $n\geq 4$. Let $R/([x,y])=C=\mathbb{Z}[x,y]/\mathfrak{a}$ as usual. Then the ideal $(p,x,y)^{n-1}$ is strictly contained in $\mathfrak{a}$.
\end{corollary}

\begin{proof}
Since $R$ is minimal noncommutative, the ideal $\mathfrak{a}$ belongs to a minimal generating pair and hence it contains $(p,x,y)^{n-1}$. However if $\mathfrak{a}=(p,x,y)^{n-1}$ then $\soc(C)=(p,x,y)^{n-2}$ which has dimension more than $3$ for $n\geq 4$. This is impossible by Proposition~\ref{IdealsInMinimalPairs}~(i).
\end{proof}

We saw in Example~\ref{Generaln=3} for $n=3$ that the procedure could end after the first step of identifying elements of $(p,x,y)^2\lhd A$ with scalar multiples of $[x,y]$. In this case, $\mathfrak{a}=(p,x,y)^2$. Corollary~\ref{MoreThanOneStep} says that this can only happen for $n=3$, and that for $n\geq 4$ the iterative procedure will never end after a single step.

In general, ideals do not satisfy the equality of Proposition~\ref{IdealsInMinimalPairs}~(iii). We give an example of an ideal for which the socle is small enough but the equality fails, and which therefore cannot belong to any minimal generating pair.

\begin{numberedexample}
For a prime $p$ and $n=6$, consider the ideal $\mathfrak{b}=(p,x^3y,xy^3)+(x,y)^5$ of $\mathbb{Z}[x,y]/(p,x,y)^6$ and let $B$ be the ring $\mathbb{Z}[x,y]/\mathfrak{b}$. The ideal $\mathfrak{b}$ satisfies the inclusion
\[(p,x,y)^{n-1}\subset\mathfrak{b}\subset(p)+(x,y)^2\]
and $\soc(B)$ has dimension three, and is generated by $x^4,y^4$ and $x^2y^2$. However, all three of $p(x^2y^2), x(x^2y^2)$ and $y(x^2y^2)$ belong to $\mathfrak{b}(p,x,y)$. Thus $x^2y^2$ is an element of the quotient ideal $\left(\mathfrak{b}(p,x,y):(p,x,y)\right)$ which does not belong to $\mathfrak{b}$.

Therefore, by Proposition~\ref{IdealsInMinimalPairs}, for any homomorphism $\nu:\mathfrak{b}\to F_p$ the pair $\left(\mathfrak{b},\nu\right)$ cannot be minimal generating. From the noncommutative perspective, if $B$ is the abelianisation of a ring $R$ then the ideal generated by $x^2y^2$ in $R$ is a minimal ideal which does not coincide with $([x,y])$, showing that $R$ cannot be minimal noncommutative.
\end{numberedexample}

\begin{problem}
For fixed prime $p$ and integer $n\geq 3$, describe all ideals $\mathfrak{b}$ in $\mathbb{Z}[x,y]/(p,x,y)^n$ which belong to a minimal generating pair.
\end{problem}

\section{Infinite minimal noncommutative rings}\label{InfiniteCase}

As we mentioned in the introduction, Bell and Danchev showed that, for the purpose of commutativity theorems, one only needs to understand finite minimal noncommutative rings. It is however natural to ask about the infinite case.

\begin{problem}\label{ProblemInfiniteMinimalNoncommutative}
Does there exist an infinite minimal noncommutative ring?
\end{problem}

As noted in \cite[Remark~2.6]{BellandDanchev}, this question ``appears to be difficult". It is pointed out for instance that, even in the case of simple rings (where the minimality condition on homomorphic images is trivially satisfied), the problem is not straightforward.

\begin{remark}
In the context of non-unital rings, noncommutative rings whose proper subrings are all commutative were studied in \cite{Ikeda_OneStepNoncommutative}, where they are called \textit{one-step noncommutative}. It is shown that a one-step noncommutative non-unital ring $R$ must be finite if $R/J$ is finite over its centre, where $J$ is the Jacobson radical of $R$. In fact, it is not known whether infinite one-step noncommutative non-unital rings exist. The result of \cite{Ikeda_OneStepNoncommutative} implies that this question reduces to the problem of the existence of a one-step noncommutative division ring. Interestingly, we will see that Problem~\ref{ProblemInfiniteMinimalNoncommutative} also reduces to the division ring case, and both problems turn out to be equivalent. Of course the condition on all subrings is very strong when one does not require rings to have an identity element.
\end{remark}

In the context of rings with $1$ that we are interested in, one can easily construct infinite one-step noncommutative rings. For instance, for a prime $p$, the ring
\[R\coloneqq\mathbb{Z}\langle x,y\rangle/\left(x^2,xy,y^2,px,py\right)\]
is not commutative, but all of its proper subrings are (the proof is similar to that of Proposition~\ref{SubringsInA_pCommutative}). Its factor ring $R/pR$ is a finite minimal noncommutative ring in the third class, described in Example~\ref{Generaln=3}. However, we will show in Corollary~\ref{OneStepNoncoHasMinimalImage} that every noncommutative ring whose proper subrings are all commutative has a minimal noncommutative homomorphic image. Thus, in our context of unital rings, one should have a condition on both subrings and homomorphic images, and the right formulation is Problem~\ref{ProblemInfiniteMinimalNoncommutative}. We start with an observation that generalises what we saw in the finite case.

\begin{proposition}\label{A22Jul22}
Let $R$ be a noncommutative ring whose proper subrings are all commutative. Then $R$ is a homomorphic image of the ring $\mathbb{Z}\langle x,y\rangle$, and the image in $R$ of the ideal $([x,y])$ of $\mathbb{Z}\langle x,y\rangle$ is the nonzero commutator ideal $[R,R]$. Moreover, $R$ is minimal noncommutative if and only if every nonzero ideal of $R$ contains $[R,R]$.
\end{proposition}

\begin{proof}
Take any two noncommuting elements $u,v$ in $R$ and consider the ring homomorphism $\mathbb{Z}\langle x,y\rangle\to R$, which maps $x$ to $u$ and $y$ to $v$. Its image is a noncommutative subring of $R$ and hence is equal to $R$ since proper subrings are commutative. The factor ring $R/([u,v])$ is a homomorphic image of $\mathbb{Z}[x,y]$, hence $R/([u,v])$ is commutative, i.e. $([u,v])=[R,R]$. The last statement follows from the definition of minimal noncommutative rings.
\end{proof}

\begin{corollary}\label{OneStepNoncoHasMinimalImage}
Let $R$ be a noncommutative ring whose proper subrings are all commutative. Then there exists an ideal $E$ of $R$ such that $R/E$ is minimal noncommutative.
\end{corollary}

\begin{proof}
By Proposition~\ref{A22Jul22}, we know that $R=\mathbb{Z}\langle x,y\rangle/I$ for some ideal $I$ which does not contain $[x,y]$. Let $\mathcal{P}$ be the set of those ideals $V$ of $R$ such that $R/V$ is not commutative. The zero ideal belongs to $\mathcal{P}$. If $\lbrace V_\alpha\rbrace$ is a chain in $\mathcal{P}$ then none of the $V_\alpha$ contains $[x,y]$ as $([x,y])=[R,R]$ by Proposition~\ref{A22Jul22}. So the union $U\coloneqq\bigcup_\alpha V_\alpha$ is an ideal of $R$ which does not contain $[x,y]$. In particular, $R/U$ is not commutative and hence $U$ is an upper bound for the chain $\lbrace V_\alpha\rbrace$ in $\mathcal{P}$. By Zorn's lemma, $\mathcal{P}$ must contain a maximal element $E$ and thus $R/E$ is minimal noncommutative.
\end{proof}

We now turn to Problem~\ref{ProblemInfiniteMinimalNoncommutative} on infinite minimal noncommutative rings. From now on, given a minimal noncommutative ring $R$, we will assume the presentation $R=\mathbb{Z}\langle x,y\rangle/I$ of Proposition~\ref{A22Jul22}. Given a ring $R$, we denote by $R^e$ its enveloping algebra $R\otimes_\mathbb{Z} R^{\text{op}}$, where $R^{\text{op}}$ is the opposite ring of $R$. We will use the fact that two-sided ideals in $R$ correspond to $R^e$-submodules of $R$.

We will now see that a minimal noncommutative PI ring is necessarily finite. This statement does not appear in \cite{BellandDanchev}, but it is implicitly showed in the proof of \cite[Theorem~2.4]{BellandDanchev}. We essentially follow their proof, with minor changes to fit the different statement.

\begin{theorem}\label{PI case}
A minimal noncommutative PI ring is finite.
\end{theorem}

\begin{proof}
Let $R$ be a minimal noncommutative PI ring. By Proposition~\ref{A22Jul22}, we know that $R=\mathbb{Z}\langle x,y\rangle/I$ for some ideal $I$, and that the ideal $([x,y])$ in $R$ is nonzero and contained in every nonzero ideal. Denote this unique minimal ideal by $L$. Then $L$ is a simple module over the enveloping algebra $R^e$. Its annihilator $Q\lhd R^e$ is therefore a primitive ideal. By Regev's tensor product theorem (see e.g. \cite[Theorem~6.1.1]{RowenBook}), the tensor product ring $R^e$ is PI too. Thus $R^e/Q$ is a primitive PI ring, and by Kaplansky's theorem (see e.g. \cite[Theorem~1.5.16]{RowenBook}) it is of the form $M_n(D)$, where $D$ is a division ring finite dimensional over its centre. Letting $b\geq 0$ be the characteristic of $R^e/Q$, we have the ring inclusions
\[\mathbb{Z}/b\mathbb{Z}\hookrightarrow \Z(R^e/Q)\hookrightarrow R^e/Q.\]
The ring $R^e/Q$ is finite dimensional over its centre $\Z(R^e/Q)$, and finitely generated as an algebra over $\mathbb{Z}/b\mathbb{Z}$. It follows by the Artin-Tate Lemma (see e.g. \cite[Lemma~13.9.10]{McConnell-Robson}) that the centre $\Z\left(R^e/Q\right)$ is a finitely generated $\mathbb{Z}$-algebra. 

However $\Z\left(R^e/Q\right)$ is a field. It is well known (e.g. from the Nullstellensatz) that a field finitely generated as an algebra over $\mathbb{Z}$ must be a finite field. Thus $\Z\left(R^e/Q\right)$ is finite, and so is the finite dimensional vector space $R^e/Q$. Since $Q$ is the annihilator of $L$ in $R^e$, this implies that $L$ is finite too. Now, the ring $R/L$ is a homomorphic image of the commutative ring $\mathbb{Z}[x,y]$, hence it is Noetherian. Since the ideal $L$ is finite, then $R$ is Noetherian too.

As $R$ is of the form $\mathbb{Z}\langle x,y\rangle/I$, it is countable. Take an enumeration $r_1, r_2,\ldots$ of all the elements of $R$. We will show that $R$ is a finite module over its centre. For each $r$ in $R$, both commutators $[x,r]$ and $[y,r]$ belong to the ideal $L$ since $R/L$ is commutative. Thus each $r$ yields a map $f_r:\lbrace x,y\rbrace\to L$ given by $f_r(x)=[x,r]$ and $f_r(y)=[y,r]$. Since $L$ is finite, there are only finitely many maps from $\lbrace x,y\rbrace$ to $L$. Hence there is a natural number $N$ such that for all $m>N$ there exists $i\leq N$ such that $f_{r_m}=f_{r_i}$. This means $[x,r_m]=[x,r_i]$ i.e. $[x,r_m-r_i]=0$. Similarly, $r_m-r_i$ commutes with $y$ too. Therefore, $r_m-r_i$ is central in $R$, and it follows that the finitely many elements $r_1,\ldots,r_N$ generate $R$ as a module over its centre $\Z(R)$.

Now, take any element $z$ in $R$. The set $\lbrace [x,y]z^i:i\geq 0\rbrace$ is contained in the ideal $L$ and is therefore finite. Thus there exist natural numbers $k<l$ such that $[x,y](z^l-z^k)=0$. Now we assume further that $z$ is central, and it follows that $L(z^l-z^k)=0$. We will show that the central element $z^l-z^k$ is nilpotent in $R$. For each $m\geq 1$, denote by $I_m$ the ideal
\[\left\lbrace x\in R: x\left(z^l-z^k\right)^m\in L\right\rbrace.\]
Then $I_1\subset I_2\subset\cdots$ is an ascending chain of ideals in $R$, which we showed is Noetherian. We then have $I_j=I_{j+1}$ for some $j\geq 1$. We assume that $(z^l-z^k)^{j+1}$ is not zero and seek a contradiction. Since $R$ is minimal noncommutative, the nonzero ideal $R(z^l-z^k)^{j+1}$ contains the minimal ideal $L$. So there exists $s\in R$ such that
\[[x,y]=s(z^l-z^k)^{j+1}=s(z^l-z^k)^j(z^l-z^k).\]
Therefore $s$ belongs to $I_{j+1}=I_j$ and hence $[x,y]$ belongs to $L(z^l-z^k)$ which is zero by our choice of $k$ and $l$. This is absurd as $R$ is not commutative. This contradiction shows that $(z^l-z^k)^{j+1}=0$. Since $z$ was an arbitrary central element, it follows that $\Z(R)$ is algebraic over $\mathbb{Z}$. Taking $z=2$, we also obtain that $R$ has finite characteristic $q>0$. Thus $R$ is a finitely generated $\mathbb{Z}/q\mathbb{Z}$-algebra and we showed that it is a finite $\Z(R)$-module. So we can again use the Artin-Tate Lemma on the ring inclusions
\[\mathbb{Z}/q\mathbb{Z}\hookrightarrow\Z(R)\hookrightarrow R\]
to obtain that $\Z(R)$ is a finitely generated $\mathbb{Z}/q\mathbb{Z}$-algebra. Since we showed that $\Z(R)$ is algebraic over $\mathbb{Z}/q\mathbb{Z}$, it follows that $\Z(R)$ is a finitely generated module over the finite ring $\mathbb{Z}/q\mathbb{Z}$. In particular $\Z(R)$ is finite. As $R$ is finitely generated as a module over its centre, it must be finite too.
\end{proof}

Thus, if an infinite minimal noncommutative ring exists, it is not PI. Since every proper homomorphic image of $R$ is commutative (hence PI), the ideal $I=0$ of $R$ is maximal with respect to the ring $R/I$ not being PI. This is well-known to imply that $I$ is prime (see e.g. \cite[Proposition~3.4]{Reyes_PIPforIdeals}). In our special case, we can give a short direct proof.

\begin{lemma}\label{InfiniteMinimalNoncommutativeIsPrime}
An infinite minimal noncommutative ring $R$ is prime.
\end{lemma}

\begin{proof}
Since $R$ is minimal noncommutative, every nonzero ideal contains $([x,y])$. So to show that $R$ is prime, it is enough to prove that the ideal $([x,y])^2$ is not zero. But if $([x,y])^2=0$ then $[X_1,X_2]X_3[X_4,X_5]$ is a polynomial identity for $R$ and $R$ is then finite by Theorem~\ref{PI case}.
\end{proof}

\begin{lemma}\label{SemiprimitiveIsDivisionRing}
Let $R$ be an infinite minimal noncommutative ring. If $R$ is semiprimitive then it is a division ring.
\end{lemma}

\begin{proof}
Assume that $R$ is semiprimitive, i.e. its Jacobson radical is zero. Then there exists a maximal right ideal $B$ that does not contain $([x,y])$. Since every nonzero ideal contains $[x,y]$ as $R$ is minimal noncommutative, it follows that $B$ contains no nonzero ideal. Therefore $R$ is right primitive and similarly it is left primitive. Then, by Jacobson's density theorem (see e.g. \cite[Corollary~0.3.7]{McConnell-Robson}), there exists a division ring $D$ such that one of the following holds:
\begin{enumerate}[topsep=0pt]
\item[(a)] either $R$ is isomorphic to $M_n(D)$ for some natural number $n$, or
\item[(b)] for every natural number $t$ there is a subring of $R$ having $M_t(D)$ as a homomorphic image.
\end{enumerate}
We will show that (a) holds with $n=1$. Suppose first that (b) holds. We seek a contradiction. By minimality of $R$, each of its proper subrings is commutative. Thus the only subring of $R$ that can have the noncommutative ring $M_2(D)$ as a homomorphic image is $R$ itself. But every proper homomorphic image of $R$ is commutative and it follows that $M_2(D)$ is isomorphic to $R$. Then $M_3(D)$ cannot be a homomorphic image of a subring of $R$, contradicting that (b) holds.

So (a) holds and $R$ is isomorphic to $M_n(D)$ for some natural number $n$ and division ring $D$. If $D$ is a commutative field, then $M_n(D)$ is PI by the Amitsur-Levitzki theorem \cite[Theorem~1]{Standidformatrices}. But $R$ cannot be PI by Theorem~\ref{PI case}. Therefore $D$ is a noncommutative division ring. Note that $R\cong M_n(D)$ contains a subring isomorphic to $D$. Since every proper subring of $R$ is commutative, it follows that $n=1$ and $R$ is isomorphic to $D$.
\end{proof}

\begin{lemma}\label{MinimalIdealFinGenThenDivisionRing}
Let $R$ be an infinite minimal noncommutative ring. If the ideal $([x,y])$ is finitely generated as a right module, then $R$ is a division ring.
\end{lemma}

\begin{proof}
If $([x,y])$ is a finitely generated right $R$-module then $([x,y])\J(R)$ is properly contained in $([x,y])$ by Nakayama's Lemma. But every nonzero ideal contains $([x,y])$ and hence the only ideal properly contained in $([x,y])$ is $0$. Thus $([x,y])\J(R)=0$. By Lemma~\ref{InfiniteMinimalNoncommutativeIsPrime} we know that $R$ is prime, and since $([x,y])\neq 0$, it follows that $R$ has zero Jacobson radical. Therefore $R$ is a division ring by Lemma~\ref{SemiprimitiveIsDivisionRing}.
\end{proof}

\begin{theorem}\label{C22July22}
Every infinite minimal noncommutative ring is a division ring.
\end{theorem}

\begin{proof}
Let $R$ be an infinite minimal noncommutative ring. Suppose that $R$ is not a division ring. We seek a contradiction. First, we prove that the minimal ideal generated by the commutator $[x,y]$ intersects trivially the centre of $R$. If $\Z(R)\cap ([x,y])$ contains a nonzero element $c$ then by Proposition~\ref{A22Jul22} we have $([x,y])=RcR=cR$. In particular, $([x,y])$ is a principal right ideal which is impossible by Lemma~\ref{MinimalIdealFinGenThenDivisionRing} as $R$ is not a division ring.

Thus $\Z(R)\cap([x,y])=0$. Note that the direct sum $\Z(R)\oplus ([x,y])$ forms a subring of $R$. If it is commutative, then every pair of elements of $([x,y])=[R,R]$ commutes, and in particular any two commutators commute. Thus $R$ satisfies the identity $\big[[X_1,X_2],[X_3,X_4]\big]$, which is impossible as $R$ is not PI by Theorem~\ref{PI case}.

Therefore $\Z(R)\oplus ([x,y])$ is a noncommutative subring of $R$. Since $R$ is minimal noncommutative this implies that $R=Z(R)\oplus ([x,y])$. So we can take two noncommuting elements $r_1,r_2$ in the ideal $([x,y])$. By Lemma~\ref{MinimalIdealFinGenThenDivisionRing}, the right module $([x,y])_R$ is not finitely generated and hence the right ideal $r_1R+r_2R$ is a proper submodule of $([x,y])_R$. Thus $\Z(R)\oplus \left(r_1R+r_2R\right)$ forms a proper subring of $R=\Z(R)\oplus ([x,y])$. It must be commutative by minimality of $R$, which is absurd as we chose noncommuting elements $r_1$ and $r_2$. This is the desired contradiction, hence $R$ must be a division ring.
\end{proof}

From now on, rather than `infinite minimal noncommutative ring' we can then say minimal noncommutative division ring (the infinite cardinality is implied by Wedderburn's little theorem that a finite division ring is a field). We reformulate Problem~\ref{ProblemInfiniteMinimalNoncommutative} accordingly.

\begin{problem}\label{ProblemMinimalNoncommutativeDivision}
Does there exist a minimal noncommutative division ring?\end{problem}

We will see that if such a division algebra exists, it exhibits very strange properties. In fact, Problem~\ref{ProblemMinimalNoncommutativeDivision} is a special case of several longstanding conjectures.

\begin{remark}[Barbaumov's problem]\label{RemarkBarbaumov}
The following question is listed in the Dniester Notebook of unsolved problems in (associative and nonassociative) ring theory, where it is attributed to Barbaumov \cite[Problem~1.14]{DniesterNotebook}. Does there exist a division algebra, infinite dimensional over its centre, in which all proper subalgebras are PI? In the special case when `PI' is replaced with `commutative', this is Problem~\ref{ProblemMinimalNoncommutativeDivision}.
\end{remark}

\begin{remark}[Latyshev problem]\label{RemarkLatyshev}
Whether an infinite division algebra can be finitely generated as a ring (i.e. as a $\mathbb{Z}$-algebra) is an open problem (see e.g. \cite[Problem~1.171]{DniesterNotebook}), sometimes referred to as the Latyshev problem. A division ring is minimal noncommutative if and only if it is generated as a ring by any pair of noncommuting elements. Thus, the existence of a minimal noncommutative division ring would yield a most extreme answer to the Latyshev problem.
\end{remark}

Though there are no known examples of infinite division algebras finitely generated over $\mathbb{Z}$, a better understood phenomenon is division rings affine over subfields. However, it is unknown whether such a division ring can be infinite dimensional. In \cite[p.412]{SkewFieldsCohn}, Cohn raises the question of whether a division ring finitely generated as an algebra over a subfield $L$ is necessarily algebraic over $L$. We will see that the existence of a minimal noncommutative division ring would give a very pathological negative answer to Cohn's problem.

\begin{lemma}\label{MaximalSubfields}
Let $R$ be a minimal noncommutative division ring with centre $\Z(R)=K$. Then
\begin{itemize}[topsep=0pt]
\item[(1)] The centraliser $C_R(a)$ of any noncentral element $a$ of $R$ is a maximal subfield of $R$.
\item[(2)] If $L$ is a maximal subfield of $R$ then $L=\C_R(a)$ for any noncentral element $a$ in $L$.
\item[(3)] If $L$ is a maximal subfield of $R$ and $S$ is any proper subring of $R$ then either $S$ is contained in $L$ or $L\cap S$ is contained in $K$.
\item[(4)] If $L$ is a maximal subfield of $R$ then $\dim_K L=\dim\left(R_L\right)=\dim\left({}_{L}R\right)=\infty$.
\end{itemize}
\end{lemma}

\begin{proof}
To prove (1), take $a$ a noncentral element of $R$. It is well known and easily checked that the centraliser $C_R(a)$ is a division ring. Since $a$ is not central in $R$, its centraliser $C_R(a)$ is properly contained in $R$. Then, by minimality of $R$, the division ring $C_R(a)$ is a field. It remains to show it is maximal as a subfield of $R$. By definition, any subring $B$ of $R$ properly containing $\C_R(a)$ will contain both $a$ and an element that does not commute with $a$. Hence $B$ must be noncommutative and therefore equal to $R$ by minimality of $R$. It follows that $\C_R(a)$ is a maximal subfield of $R$.

For (2), the maximal subfield $L$ has to contain a noncentral element $a$. Since $L$ is commutative, it is a subfield of $\C_R(a)$. If the inclusion is proper, then $C_R(a)=R$ which is impossible as $a$ is not central. Therefore $L=C_R(a)$.

For (3), suppose that $L\cap S$ contains a noncentral element $a$. Since $S$ is commutative, it is contained in $\C_R(a)$. But $\C_R(a)=L$ by (2).

(4) It is well known that if one of the dimensions is $d<\infty$ then they are all equal to $d$ and $\dim_K R=d^2$. In particular $R$ is then PI which is impossible by Theorem~\ref{PI case}.
\end{proof}

\begin{proposition}\label{AlgebraicIsCentral}
Let $R$ be a minimal noncommutative division ring. Then no noncentral element of $R$ is algebraic over the centre of $R$.
\end{proposition}

\begin{proof}
Denote by $K$ the centre of $R$. Suppose there is an element $a$ in $R\setminus K$ which is algebraic over $K$. Then $K[a]$ is a finite field extension of $K$. Since $K$ is central, we have $\C_R(K[a])=\C_R(a)$, which is a maximal subfield of $R$ by Lemma~\ref{MaximalSubfields}~(1). In particular $\C_R(a)=\C_R\left(\C_R(a)\right)$. But by the Double Centraliser Theorem (see e.g. \cite[Theorem~15.4]{FirstCourseLam}), it follows that
\[K[a]=\C_R(\C_R(K[a]))=\C_R(\C_R(a))=\C_R(a).\]
Therefore, $K[a]$ is a maximal subfield of $R$. This is impossible by Lemma~\ref{MaximalSubfields}~(4) since $K[a]$ is finite dimensional over $K$.
\end{proof}

So a minimal noncommutative division ring would give an example of a division algebra generated by only $2$ elements as an algebra over its centre, but without a single algebraic element.

\begin{definition}
Let $L$ be a subfield of a division ring $R$, and let $L^\ast$ and $R^\ast$ denote their respective multiplicative groups. Then $L$ is called \textit{self-invariant} if the normaliser $\N_{R^\ast}(L^{\ast})$ coincides with $L^\ast$.
\end{definition}

In \cite[Question~1.2]{Self-InvariantMaximalSubfields}, it is asked whether a division ring whose maximal subfields are all self-invariant must be commutative. This question is motivated by its connection with Albert's conjecture and a weak version of the Kurosh problem for division rings. The authors went on to construct a counterexample in \cite[Theorem~4.4]{Self-InvariantMaximalSubfields}, and the question is still open in the centrally finite case. We will show that again, a minimal noncommutative division ring would exhibit pathological behaviour, and constitute another counterexample.

\begin{lemma}\label{VectorSpaceOverMaximalSubfield}
Let $R$ be a minimal noncommutative division ring, and $L$ be a noncentral subfield of $R$. Then for any element $a$ in $R\setminus C_R(L)$, the $L$-vector space $L+La+La^2+\ldots$ is not closed under multiplication, i.e. it does not form a subring of $R$.
\end{lemma}

\begin{proof}
Suppose that $S\coloneqq L+La+La^2+\ldots$ forms a subring of $R$. We seek a contradiction. Since $a$ does not commute with $L$, the subring $S$ is noncommutative and is therefore equal to $R$ by minimality of $R$. In particular, $a^{-1}$ belongs to $S=R$, i.e. there are elements $\alpha_0,\ldots,\alpha_n$ in $L$ such that
\[a^{-1}=\alpha_0+\alpha_1a+\ldots+\alpha_na^n.\]
Thus $a$ is left algebraic over $L$ as
\[0=-\alpha_n^{-1}+\alpha_n^{-1}\alpha_0a+\alpha_n^{-1}\alpha_1a^2+\ldots+a^{n+1}\]
and it follows that $R=L+La+\ldots+La^n$ has $\dim\left({}_{L}R\right)=n+1$. In particular, $R$ is finite dimensional as a left vector space over the maximal subfield $C_R(L)$, which is impossible by Lemma~\ref{MaximalSubfields}~(4).
\end{proof}

\begin{proposition}\label{MaximalSubfieldsNormal}
Let $R$ be a minimal noncommutative division ring. Then every maximal subfield of $R$ is self-invariant.
\end{proposition}

\begin{proof}
Let $L$ be a maximal subfield of $R$. Take an arbitrary element $a$ in the normaliser $\N_{R^\ast} (L^\ast)$. Then $aL= La$ and it follows that the set $L+La+La^2+\ldots$ is closed under multiplication, hence it forms a subring of $R$. By Lemma~\ref{VectorSpaceOverMaximalSubfield}, this is only possible if $a$ belongs to $L$. Since $a$ was arbitrary in $\N_{R^\ast}(L^\ast)$, we have shown that $\N_{R^\ast}(L^\ast)=L^\ast$, i.e. the subfield $L$ is self-invariant.
\end{proof}

Therefore, for any maximal subfield $L$ and $a\in R\setminus L$, there exists $b\in L$ such that $aba^{-1}\not\in L$. In fact, this is the case for every noncentral $b$ in $L$.

\begin{corollary}\label{ConjugacyMovesSubfields}
Let $R$ be a minimal noncommutative division ring and $L$ be a maximal subfield of $R$. For any elements $a\in R\setminus L$ and $b\in L$, the following are equivalent.
\begin{enumerate}[topsep=0pt]
\item[(i)] $aba^{-1}$ belongs to $L$;
\item[(ii)] $aba^{-1}$ commutes with $b$;
\item[(iii)] $b$ is central in $R$.
\end{enumerate}
\end{corollary}

\begin{proof}
That (iii) implies (ii) is trivial. For (ii) implies (i), we consider two cases. First, if $b$ is central then $aba^{-1}=b$ and both (ii) and (i) are always true. In the case that $b$ is not central then, by Lemma~\ref{MaximalSubfields}~(2), we have $L=C_R(b)$ and therefore (ii) implies (i). Finally, we prove that (i) implies (iii). Since $a$ does not belong to $L$, we know by Proposition~\ref{MaximalSubfieldsNormal} that $aLa^{-1}\neq L$. Moreover, it is easily checked that $aLa^{-1}$ forms a proper subring of $R$. Then, by Lemma~\ref{MaximalSubfields}~(3), the intersection $aLa^{-1}\cap L$ is contained in $\Z(R)$. In particular, $aba^{-1}=z$ for some central element $z\in\Z(R)$. Then $ab=za=az$ and hence $b=z$ as required.
\end{proof}

In other words, for noncommuting elements $a,b\in R$, the conjugate $aba^{-1}$ never commutes with $b$. Of course, it does not commute with $a$ either (this is true in general in any group).

\begin{corollary}
Let $R$ be a minimal noncommutative division ring. For any two elements $a$ and $b$ in $R$, either $a$ and $b$ commute or
\begin{itemize}[topsep=0pt]
\item[(1)] the commutator $[a,b]$ commutes with neither $a$ nor $b$, and
\item[(2)] the multiplicative commutator $aba^{-1}b^{-1}$ commutes with neither $a$ nor $b$.
\end{itemize}
\end{corollary}

\begin{proof}
(1) Let $a,b$ be two noncentral elements of $R$ such that $[a,b]$ commutes with $a$. We will show that $a$ commutes with $b$. The proof in the case where $[a,b]$ commutes with $b$ is identical, with the roles of $a$ and $b$ reversed. Since $a$ is not central then by Lemma~\ref{MaximalSubfields}~(1), its centraliser $C_R(a)$ is a field. Thus any element $r$ in $\C_R(a)$ also commutes with $[a,b]$. Hence, the Jacobi identity
\[[r,[a,b]]+[a,[b,r]]+[b,[r,a]]=0\]
implies that $[a,[b,r]]=0$, i.e. $[b,r]$ also belongs to $\C_R(a)$. As $r$ was an arbitrary element of $C_R(a)$, it follows that $[b,\C_R(a)]$ is contained in $\C_R(a)$. Then $b\C_R(a)$ is contained in $\C_R(a)+ \C_R(a)b$ and therefore
\[C_R(a)+C_R(a)b+C_R(a)b^2+\ldots\]
forms a subring of $R$. By Lemma~\ref{VectorSpaceOverMaximalSubfield} this is impossible unless $b$ commutes with $a$.

(2) Let $a,b$ be two noncentral elements of $R$ such that $aba^{-1}b^{-1}$ commutes with $b$. Then the maximal subfield $C_R(b)$ also contains $(aba^{-1}b^{-1})b=aba^{-1}$. By Corollary~\ref{ConjugacyMovesSubfields}, since $b$ is not central this is impossible unless $a$ commutes with $b$ as required. The proof is the same if $aba^{-1}b^{-1}$ commutes with $a$ instead, using that $bab^{-1}$ then belongs to $C_R(a)$.
\end{proof}

\begin{remark}[Makar-Limanov conjecture]\label{RemarkMakar-Limanov}
The conjecture that is most opposed to the existence of a minimal noncommutative division ring is due to Makar-Limanov. Let $R$ be a division ring finitely generated (as a division algebra) and infinite dimensional over its centre $K$. In \cite[p.~284]{Makar-Limanov1984}, Makar-Limanov conjectured that $R$ must contain a free $K$-algebra in two generators. This was originally motivated by his proof in \cite{Makar-Limanov_DivisionWeyl} that the quotient division ring of the first Weyl algebra contains such a free algebra. In its full generality, the Makar-Limanov conjecture implies the Kurosh problem for division rings since if $R$ contains a free $K$-subalgebra, it cannot be algebraic over $K$. The conjecture has attracted a lot of interest and it has been proved to hold in many special cases, often with the assumption that $K$ is uncountable (see e.g. \cite{BellRogalski_FreeSubalgebrasDivision}).

These results do not apply to the case we are interested in, a minimal noncommutative ring being countable by Proposition~\ref{A22Jul22}. In fact, a minimal noncommutative division ring could be considered as strong a counterexample to the Makar-Limanov conjecture as one can imagine. Indeed, a minimal noncommutative division ring $R$ is an infinite dimensional finitely generated algebra over its centre. Moreover, the condition that each of its proper subrings is commutative not only means that $R$ does not contain a copy of the free algebra on two generators, but that every pair of noncommuting elements generates the whole of $R$ over $\mathbb{Z}$.
\end{remark}

\paragraph{Acknowledgements.} This work was conducted as part of the second author's PhD thesis at the University of Sheffield, supported by EPSRC grant EP/T517835/1. The second author would also like to thank Jason Bell for interesting discussions on the topic of this paper.

\begin{tabular}{l  l}
V. V. Bavula & \quad \quad\quad \quad \quad \quad Nathan Blacher \\ 
School of Mathematical and & \quad \quad\quad \quad\quad \quad School of Mathematical and\\
Physical Sciences &\quad \quad\quad \quad\quad \quad Physical Sciences\\
University of Sheffield & \quad \quad\quad \quad\quad \quad University of Sheffield\\
Hicks Building & \quad \quad\quad \quad\quad \quad Hicks Building\\
Sheffield S3 7RH & \quad \quad\quad \quad\quad \quad Sheffield S3 7RH\\
United Kingdom & \quad \quad\quad \quad\quad \quad United Kingdom\\
email: v.bavula@sheffield.ac.uk & \quad \quad\quad \quad\quad \quad email: nlblacher@gmail.com\\
\end{tabular}

\end{document}